\documentclass[10pt]{article}
\usepackage[utf8]{inputenc}
\usepackage{xcolor}
\usepackage{amsmath,amsfonts,amssymb,amsthm,enumitem,tensor,mathtools,graphicx,epstopdf,pgfplotstable,tikz,subcaption,xspace}
\usepackage{bm}
\usepackage{adjustbox}
\usepackage[breaklinks,bookmarks=false]{hyperref}
\hypersetup{colorlinks, linkcolor=blue, citecolor=blue,
urlcolor=blue, plainpages=false, pdfwindowui=false,
pdfstartview={FitH}}

\usepackage{cleveref}
\usepackage[backend=biber,style=numeric,sorting=nty]{biblatex}
\definecolor{plum}{HTML}{92268F}
\definecolor{fgreen}{HTML}{009B55}

\newlength\smallfigureheight
\newlength\smallfigurewidth
\pgfplotsset{
	scale only axis,
	width=\smallfigurewidth,
	height=\smallfigureheight,
	every axis/.append style={
		xticklabel style={font=\footnotesize},
		xlabel style={font=\footnotesize},
		yticklabel style={font=\footnotesize},
		ylabel style={font=\footnotesize},
		title style={font=\footnotesize},
		legend image post style={scale=0.5},
		legend style={font=\footnotesize}
	}
}

\newcommand{\logLogHalfSlopeTriangle}[2]{
	\pgfplotsextra{
		\pgfkeys{/pgf/fpu}
		\pgfmathsetmacro{\xC}{#1}
		\pgfmathsetmacro{\yC}{#2}
		\pgfmathsetmacro{\scale}{5}  	
		\pgfmathsetmacro{\xA}{\xC*\scale}
		\pgfmathsetmacro{\yA}{\yC*sqrt(\scale)}
		\pgfmathsetmacro{\xB}{\xA}
		\pgfmathsetmacro{\yB}{\yC}       
		\pgfkeys{/pgf/fpu=false}
		\coordinate (A) at (\xA,\yA);
		\coordinate (B) at (\xB,\yB);
		\coordinate (C) at (\xC,\yC);
		\draw   (A) -- (B) node[midway,right]{{\scriptsize $\tfrac{1}{2}$}} -- (C)  node[midway,below]{{\scriptsize $1$}}--  cycle;
}}

\newcommand{\IlogLogSlopeTriangle}[2]{
	\pgfplotsextra{
		\pgfkeys{/pgf/fpu}
		\pgfmathsetmacro{\xC}{#1}
		\pgfmathsetmacro{\yC}{#2}
		\pgfmathsetmacro{\scale}{3}  	
		\pgfmathsetmacro{\xA}{\xC*\scale}
		\pgfmathsetmacro{\yA}{\yC*\scale}
		\pgfmathsetmacro{\xB}{\xA}
		\pgfmathsetmacro{\yB}{\yC}       
		\pgfkeys{/pgf/fpu=false}
		\coordinate (A) at (\xA,\yA);
		\coordinate (B) at (\xB,\yB);
		\coordinate (C) at (\xC,\yC);
		\draw   (A) -- (B) node[midway,right]{{\scriptsize $1$}} -- (C)  node[midway,below]{{\scriptsize $1$}}--  cycle;
}}

\newcommand{\IbiglogLogSlopeTriangle}[2]{
	\pgfplotsextra{
		\pgfkeys{/pgf/fpu}
		\pgfmathsetmacro{\xC}{#1}
		\pgfmathsetmacro{\yC}{#2}
		\pgfmathsetmacro{\scale}{5}  	
		\pgfmathsetmacro{\xA}{\xC*\scale}
		\pgfmathsetmacro{\yA}{\yC*\scale}
		\pgfmathsetmacro{\xB}{\xA}
		\pgfmathsetmacro{\yB}{\yC}       
		\pgfkeys{/pgf/fpu=false}
		\coordinate (A) at (\xA,\yA);
		\coordinate (B) at (\xB,\yB);
		\coordinate (C) at (\xC,\yC);
		\draw   (A) -- (B) node[midway,right]{{\scriptsize $1$}} -- (C)  node[midway,below]{{\scriptsize $1$}}--  cycle;
}}

\newcommand{\logLogTSlopeTriangle}[2]{
	\pgfplotsextra{
		\pgfkeys{/pgf/fpu}
		\pgfmathsetmacro{\xC}{#1}
		\pgfmathsetmacro{\yC}{#2}
		\pgfmathsetmacro{\scale}{3}  	
		\pgfmathsetmacro{\xA}{\xC*\scale}
		\pgfmathsetmacro{\yA}{\yC*\scale^2}
		\pgfmathsetmacro{\xB}{\xA}
		\pgfmathsetmacro{\yB}{\yC}       
		\pgfkeys{/pgf/fpu=false}
		\coordinate (A) at (\xA,\yA);
		\coordinate (B) at (\xB,\yB);
		\coordinate (C) at (\xC,\yC);
		\draw   (A) -- (B) node[midway,right]{{\scriptsize $2$}} -- (C)  node[midway,below]{{\scriptsize $1$}}--  cycle;
}}

\newcommand{\biglogLogTSlopeTriangle}[2]{
	\pgfplotsextra{
		\pgfkeys{/pgf/fpu}
		\pgfmathsetmacro{\xC}{#1}
		\pgfmathsetmacro{\yC}{#2}
		\pgfmathsetmacro{\scale}{5}  	
		\pgfmathsetmacro{\xA}{\xC*\scale}
		\pgfmathsetmacro{\yA}{\yC*\scale^2}
		\pgfmathsetmacro{\xB}{\xA}
		\pgfmathsetmacro{\yB}{\yC}       
		\pgfkeys{/pgf/fpu=false}
		\coordinate (A) at (\xA,\yA);
		\coordinate (B) at (\xB,\yB);
		\coordinate (C) at (\xC,\yC);
		\draw   (A) -- (B) node[midway,right]{{\scriptsize $2$}} -- (C)  node[midway,below]{{\scriptsize $1$}}--  cycle;
}}

\newtheorem{definition}{Definition}
\newtheorem{theorem}{Theorem}
\newtheorem{corollary}{Corollary}
\newtheorem{remark}{Remark}
\newtheorem{lemma}{Lemma}

\newtheorem{eg}{Example}

\newcommand{\conv}{\overline{\mathrm{conv}}\,}
\newcommand{\abs}[1]{\lvert #1 \rvert}

\title{Analysis and Numerical Approximation of Stationary Second-Order Mean Field Game Partial Differential Inclusions\footnote{University College London, Department of Mathematics, 25 Gordon Street, London, WC1H 0AY, UK; yohance.osborne.16@ucl.ac.uk, i.smears@ucl.ac.uk}}
\author{Yohance A. P. Osborne and 
Iain Smears}

\pgfplotsset{compat=newest}

\usetikzlibrary{calc}

\begin{document}
	
\maketitle
\pagenumbering{arabic}

\begin{abstract}
The formulation of Mean Field Games (MFG) typically requires continuous differentiability of the Hamiltonian in order to determine the advective term in the Kolmogorov--Fokker--Planck equation for the density of players. However, in many cases of practical interest, the underlying optimal control problem may exhibit bang-bang controls, which typically lead to nondifferentiable Hamiltonians.
We develop the analysis and numerical analysis of stationary MFG for the general case of convex, Lipschitz, but possibly nondifferentiable Hamiltonians.
In particular, we propose a generalization of the MFG system as a Partial Differential Inclusion (PDI) based on interpreting the derivative of the Hamiltonian in terms of subdifferentials of convex functions.
We establish existence of a weak solution to the MFG PDI system, and we further prove uniqueness under a similar monotonicity condition to the one considered by Lasry and Lions. 
We then propose a monotone finite element discretization of the problem, and we prove strong $H^1$-norm convergence of the approximations to the value function and strong $L^q$-norm convergence of the approximations of the density function. We illustrate the performance of the numerical method in numerical experiments featuring nonsmooth solutions.
\end{abstract}
\medskip

\begin{small} 
  \medskip
  
\end{small}

\section{Introduction}

Mean Field Games (MFG), as introduced by Lasry \& Lions \cite{lasry2006jeux1,lasry2006jeux,lasry2007mean} and independently by Huang, Caines \& Malham\'e~\cite{huang2006large}, consider the asymptotic behaviour of rational stochastic differential games as the number of players approaches infinity. 
Under suitable assumptions, the system of equations consists of a Hamilton--Jacobi--Bellman (HJB) equation for the value function associated to the underlying stochastic optimal control problem faced by the players, coupled with a Kolmogorov--Fokker--Planck (KFP) equation for the density of players within the state space of the game.
MFG systems find applications in a broad range of areas, such as economics, population dynamics, and mass transport~\cite{AchdouEtAl2014,GueantLasryLions2003,GomesSaude2014}.
We refer the reader to the surveys in~\cite{achdou2020mean,GomesSaude2014,gomes2016regularity} for extensive reviews of the literature on the theory and applications for a variety of MFG problems.

The numerical solution of MFG systems is an active area of research and has led to various approaches.
Monotone finite difference methods on Cartesian grids are considered in~\cite{achdou2010mean,achdou2013mean,achdou2016convergence}.
In particular, under the assumption that the continuous problem admits a unique classical solution, \cite{achdou2013mean} shows the convergence of the approximation of the value function in some first-order Sobolev space for the stationary case, and in some Bochner--Sobolev space for the time-dependent case, along with convergence of the approximations of the density function in some Lebesgue spaces.
The assumption of the existence of a classical solution was then removed in~\cite{achdou2016convergence}, which showed convergence of the approximations to a weak solution of the system.
There is also an alternative approach to the solution of the problem when the couplings of the system are local. In this case, the MFG system can at least formally be related to the first-order optimality conditions of convex optimization problems, which leads to other methods based on optimization, see for example \cite{andreev2017preconditioning,AriasKaliseSilva2018}.
Fully discrete semi-Lagrangian schemes have also been proposed in~\cite{CarliniSilva14,CarliniSilve15} for first- and second-order MFG systems.


We now outline the motivation for the present paper. Recall that MFG PDE systems are derived from models of large numbers of players solving stochastic optimal control problems. It is well-known from stochastic optimal control that, in many applications of practical interest, the underlying controls may be of \emph{bang-bang} type, which typically lead to discontinuities in the optimal control policies and the possibility of nonunique optimal controls in some regions of the state space.
In turn, this generally leads to nondifferentiable Hamiltonians, which pose special challenges for the analysis and numerical analysis of MFG systems.

To illustrate these challenges, we consider as a model problem a stationary MFG system of the form
\begin{subequations}\label{sys1}
	\begin{align}
		- \nu \Delta u+H(x,\nabla u) + \kappa u& =F[m]  &\text{ in }\Omega,& \label{originalform1}
		\\ 
		- \nu\Delta m-\text{div}\left(m\frac{\partial H}{\partial p}(x,\nabla u)\right)+\kappa m & =G(x) &\text{ in }\Omega,&\label{originalform2} 
	\end{align}
\end{subequations}
along with homogeneous Dirichlet boundary conditions $u=0$ and  $m=0$ on $\partial \Omega$.
The unknowns $u$ and $m$ denote, respectively, the value function and the density function for the player distribution of the game.
Here, the domain $\Omega\subset\mathbb{R}^n$ is a bounded connected open set in $\mathbb{R}^n$,  $n\geq 2$, and $\nu>0$ and $\kappa\geq 0$ are constants. 
Precise assumptions on the data $H$, $F$, and $G$ are given below in Section~\ref{Sec2}.
The system \eqref{sys1} includes as special cases the stationary MFG model considered in \cite{lasry2006jeux1} (in which case $\kappa$ and $G$ vanish) and some models of discounted MFG~\cite{ferreira2019existence}. 
However, note that in contrast to the periodic boundary conditions considered in~\cite{lasry2006jeux1}, we consider \eqref{sys1} along with Dirichlet boundary conditions, which arise in models where players may enter or exit the game, and thus $m$ is not a probability density function in general.
This explains why there is no Lagrange multiplier term in the first equation~\eqref{originalform1}.
The source term $G$ and the term involving $\kappa$ in~\eqref{sys1} are also relevant in the context of temporal semi-discretizations of time-dependent MFG systems.

The Hamiltonian in~\eqref{sys1} is given in terms of components of the underlying stochastic optimal control problem; we therefore consider Hamiltonians $H$ of the form
\begin{equation}\label{1''}
	H(x,p):=\sup_{\alpha\in\mathcal{A}}\left(b(x,\alpha)\cdot p-f(x,\alpha)\right)\quad \forall (x,p)\in\overline{\Omega}\times\mathbb{R}^n,
\end{equation}
where $\mathcal{A}$ denotes the set of controls, and where $b $ is the opposite of the controlled drift and $f$ is a control-dependent running cost component set by the underlying stochastic optimal control problem.
For simplicity, we assume that $\mathcal{A}$ is a compact metric space, and  $b: \overline{\Omega}\times\mathcal{A}\to\mathbb{R}^n$, $f:\overline{\Omega}\times\mathcal{A}\to\mathbb{R}$ are uniformly continuous, so that the supremum in~\eqref{1''} is achieved.
In many applications, the controls that achieve the supremum~\eqref{1''} may be nonunique for some $(x,p)\in \Omega\times \mathbb{R}^n$, which often leads to discontinuous optimal controls of \emph{bang-bang} type. In these cases, the Hamiltonian $H$ is then typically Lipschitz continuous but not differentiable everywhere.
However, most works so far on MFG require differentiable or even $C^1$ Hamiltonians, which can be quite restrictive in practice.

Nondifferentiable Hamiltonians pose an immediate and obvious challenge for analysis since the advective term in~\eqref{originalform2} is then no longer well-defined in a classical sense.
This leads to the problem of finding a suitable relaxed meaning for the equation in these situations. From a modelling perspective, this corresponds to the question of how the players of the game choose among the optimal controls when they are not unique.
To the best of our knowledge, the analysis of MFG with nondifferentiable Hamiltonians seems to only have been considered in~\cite{ducasse2020second} for the special case of Hamiltonians of the form $\overline{H}(m,p)=\overline{\kappa}(m)\lvert p \rvert$, for some given function $\overline{\kappa}$; see Remark~\ref{rem:ducasse} for further comments. Specific examples are presented in~\cite{BardiFischer19} showing how uniqueness of solutions may fail for nondifferentiable Hamiltonians and nonmonotone couplings; in these examples, the advective term is unambiguous since the gradients of the value functions avoid the points of nondifferentiability of the Hamiltonian.
Otherwise, the analysis and numerical analysis of MFG problems with nondifferentiable Hamiltonians remains largely untouched.

Our first main contribution in this work is to provide a suitable generalized meaning for the system~\eqref{sys1} when $H$ is nondifferentiable, and to prove results on the existence and uniqueness of solutions under conditions where they are expected to hold. 
Using the fact that the Hamiltonian $H$ is convex with respect to its second argument, our approach is based on relaxing \eqref{originalform2} as the following \emph{Partial Differential Inclusion}~(PDI):
\begin{equation}\label{FKgood}
	-\nu\Delta m +\kappa m- G(x)\in \text{div}\left(m\partial_p H(x,\nabla u)\right)\text{ }\text{ in }\Omega,
\end{equation} 
where $\partial_p H$ denotes the Moreau--Rockafellar pointwise partial subdifferential of $H$ with respect to $p$, and where the inclusion is understood in a suitable weak sense.
The resulting MFG PDI is then
\begin{equation}\label{sys}
	\begin{aligned}
		- \nu\Delta u+H(x,\nabla u)+\kappa u&=F[m]   &&\text{ in }\Omega,
		\\
		-\nu\Delta m +\kappa m- G(x)&\in \text{div}\left(m\partial_p H(x,\nabla u)\right) &&\text{ in }\Omega,
		\\
		u=0,\qquad m&=0, & &\text{ on }\partial\Omega. 
	\end{aligned}
\end{equation}
We first prove existence of weak solutions of~\eqref{sys} for rather general problem data. 
Then, crucially, we show uniqueness of solutions for~\eqref{sys} for monotone couplings following the strategy of Lasry \&~Lions~\cite{lasry2006jeux1,lasry2007mean}, thus extending important uniqueness results to the case of nondifferentiable Hamiltonians.
Our approach is also significant in terms of the mathematical modelling, since it does not require additional modelling assumptions on how the players choose among the optimal controls when they are not unique; see Remark \ref{rem:drift_meaning} below for more specific comments.

Our second main contribution is to propose and study a monotone Finite Element Method (FEM) for approximating weak solutions to the MFG PDI \eqref{sys}.
In this context, monotonicity of the FEM refers to the presence of a discrete maximum principle. There is a wide range of approaches to constructing monotone FEMs; see for instance \cite{ciarlet1973maximum,baba1981conservation,MIZUKAMI1985181,xu1999monotone,burman2002nonlinear}. The discretization considered here is based on the one from~\cite{JensenSmears2013} for degenerate fully nonlinear HJB equations, where convergence to the unique viscosity solution was shown; see also~\cite{Jensen2017,JaroszkowskiJensen2022}. To keep the analysis as simple as possible, we concentrate on a monotone FEM where the discrete maximum principle is achieved via artificial diffusion on strictly acute meshes.

The main result on the analysis of the  numerical approximations is given in Theorem \ref{convergence}, which shows convergence of the numerical approximations in the small mesh limit for uniquely solvable MFG PDI systems.
In particular, we prove strong convergence in the $H^1$-norm of the approximations to the value function $u$.
We also show that the approximations of the density function $m$ converge strongly in $L^q$-norms for $q\in [1,2^*)$, $2^*=\frac{2n}{n-2}$, as well as weak convergence in $H^1$.
For general nondifferentiable Hamiltonians, a proof of strong convergence in $H^1$ of the density approximations is not currently available owing to a lack of continuity in the advective terms of the KFP equation. However, if some additional continuity is assumed (which holds for instance when the Hamiltonian is $C^1$), then the density approximations also converge strongly in $H^1$, see Corollary \ref{convergence-cor}.
We then complement the convergence analysis with two numerical experiments that illustrate the performance of the method.

This paper is organized as follows. We outline the notation in Section~\ref{Sec2}, and we formulate the notion of weak solution for the MFG PDI \eqref{sys} and state the main results on the continuous problem in Section~\ref{Sec3}.
The main results on the analysis of the continuous problem are then proved in Section~\ref{Sec 4}.
In Section~\ref{Sec5} we introduce a monotone finite element scheme along with main results on the well-posedness of the method and its convergence.
This is followed by the proofs of these results in Section~\ref{Sec6}.
Section~\ref{Sec 7} presents the results of some numerical experiments.

\section{Notation}\label{Sec2}

We denote $\mathbb{N}\coloneqq \{1,2,3,\cdots\}$, and let $n\in \mathbb{N}$, $n\geq 2$.
For a Lebesgue measurable set $\omega \subset \mathbb{R}^n$, let $\lVert \cdot \rVert_\omega$ denote the standard $L^2$-norm for scalar- and vector-valued functions on $\omega$.
Let $\Omega$ be a bounded, open connected subset of $\mathbb{R}^n$ with Lipschitz boundary $\partial \Omega$.
The $n$-dimensional open ball of radius $r$ and centre $x_0\in\mathbb{R}^n$ is denoted by $B_r(x_0)$.
For a set $\mathcal{C}\subset \mathbb{R}^n$, we denote its closed convex hull by $\conv\mathcal{C}$.

We make the following assumptions on the data appearing in \eqref{sys}.
Let $\nu>0$ and $\kappa\geq 0$ be constants, and let $G\in H^{-1}(\Omega)$. We will say that $G\in H^{-1}(\Omega)$ is nonnegative in the sense of distributions if 
$\langle G, \phi\rangle_{H^{-1}\times H_0^1} \geq 0$ for all functions $\phi\in H_0^1(\Omega)$ that are nonnegative a.e.\ in $\Omega$. Next, let $F:L^2(\Omega)\to H^{-1}(\Omega)$ be a possibly nonlocal operator that satisfies 
\begin{subequations}
	\begin{align} 
		\|F[z]\|_{H^{-1}(\Omega)}&\leq c_1\left(\|z\|_{\Omega}+1\right)&&\forall z\in L^2(\Omega)\label{F1},
		\\ \|F[m_1]-F[m_2]\|_{H^{-1}(\Omega)}&\leq c_2\|m_1-m_2\|_{\Omega}&&\forall m_1,m_2\in  L^2(\Omega)\label{F2}  
	\end{align}
\end{subequations}
where $c_1,c_2\geq 0$ are constants. We will say that $F$ is strictly monotone if 
\begin{equation}\label{eq:F_monotone}
	\langle F[m_1]-F[m_2],m_1-m_2\rangle_{H^{-1}\times H_0^1} \leq 0\Longrightarrow m_1=m_2
\end{equation}
whenever $m_1,m_2\in H_0^1(\Omega)$. 
Note that although the domain of $F$ is $L^2(\Omega)$, the monotonicity condition~\eqref{eq:F_monotone} is needed only for arguments in the smaller space $H^1_0(\Omega)$.

\begin{eg}
	\emph{The conditions \eqref{F1}, \eqref{F2} and \eqref{eq:F_monotone} are satisfied by a broad class of operators. For example, this class includes local operators $F:L^2(\Omega)\to L^2(\Omega)$ of the form $F[z](x)\coloneqq f(z(x))$, $x\in\Omega$, $z\in L^2(\Omega)$, where the function $f:\mathbb{R}\to\mathbb{R}$ is strictly monotone and Lipschitz continuous.
	This class also includes some nonlocal operators, such as $F\coloneqq (-\Delta)^{-1}: L^2(\Omega)\to H_0^1(\Omega)$ where $(-\Delta)^{-1}$ denotes the inverse Laplacian with a homogeneous Dirichlet boundary condition.
	In this case $F$ is strongly monotone with respect to the $H^{-1}(\Omega)$-norm, and is thus strictly monotone in $L^2(\Omega)$.
	The conditions above also allow some operators of differential-type.
	For instance, we can have $F:L^2(\Omega)\to H^{-1}(\Omega)$ defined by $\langle F[z],\phi \rangle_{H^{-1}\times H_0^1}\coloneqq -\int_{\Omega}z\bm{v}{\cdot}\nabla \phi\,\mathrm{d}x$ for all $z\in L^2(\Omega)$ and $\phi\in H_0^1(\Omega)$, where $\bm{v}\in C^{1}(\overline{\Omega};\mathbb{R}^n)$ is a vector field that satisfies $\nabla{\cdot}\bm{v}>0$ in $\Omega$.
	In this case $\langle F[m_1]-F[m_2],m_1-m_2\rangle_{H^{-1}\times H_0^1}=\frac{1}{2}\int_\Omega (\nabla{\cdot}\bm{v})(m_1-m_2)^2\mathrm{d}x $ for all $m_1$ and $m_2$ in $H^1_0(\Omega)$, so $F$ is strictly monotone on $H^1_0(\Omega)$.
	This is an example where it is helpful to require monotonicity only on the smaller space $H^1_0(\Omega)$.}
\end{eg}

Recall that the Hamiltonian $H\colon \overline{\Omega}\times \mathbb{R}^n \rightarrow \mathbb{R}$ is defined by~\eqref{1''}, and that $b$ and $f$ are uniformly continuous on $\overline{\Omega}\times \mathcal{A}$ with $\mathcal{A}$ a compact metric space.
The Hamiltonian $H$ then satisfies the following bounds 
\begin{subequations}\label{bounds}
	\begin{align} 
		|H(x,p)|&\leq c_3\left(|p|+1\right)&&\forall(x,p)\in \overline{\Omega}\times\mathbb{R}^n, \label{bounds:growth}\\ 
		|H(x,p)-H(x,q)|&\leq c_4|p-q|&&\forall(x,p,q)\in\overline{\Omega}\times\mathbb{R}^n\times\mathbb{R}^n,\label{bounds:lipschitz}
	\end{align} 
\end{subequations}
with
$c_3\coloneqq \max\left\{\|b\|_{C(\overline{\Omega}\times\mathcal{A};\mathbb{R}^n)},\|f\|_{C(\overline{\Omega}\times\mathcal{A})}\right\}$ and $ c_4\coloneqq\|b\|_{C(\overline{\Omega}\times\mathcal{A};\mathbb{R}^n)}.$ 
It is then clear that the mapping $v\mapsto H(\cdot,\nabla v)$ is Lipschitz continuous from $H^1(\Omega)$ into $L^2(\Omega)$. 

Given arbitrary sets $A$ and $B$, an operator $\mathcal{M}$ that maps each point $x\in A$ to a \emph{subset} of $B$ is called a \emph{set-valued map from $A$ to $B$}, and we write $\mathcal{M}:A\rightrightarrows B$. For the Hamiltonian given by \eqref{1''} its pointwise Moreau--Rockafellar partial subdifferential with respect to $p$ is the set-valued map $ \partial_p H\colon\Omega\times\mathbb{R}^n\rightrightarrows\mathbb{R}^n$ defined by 
\begin{equation}\label{subdifferential}
	\partial_p H(x,p)\coloneqq\left\{\tilde{b}\in\mathbb{R}^n:H(x,q)\geq H(x,p)+\tilde{b}\cdot(q-p)\quad\forall q\in\mathbb{R}^n\right\}.
\end{equation}
Note that $\partial_p H(x,p)$ is nonempty for all $x\in\Omega$ and $p\in\mathbb{R}^n$ because $H$ is real-valued and convex in $p$ for each fixed $x\in\Omega$. 
Note also that for the special case of a differentiable convex function, the (partial) subdifferential at a point is simply the singleton set containing the value of the (partial) derivative at the point.
Furthermore, the subdifferential $\partial_p H$ is uniformly bounded since~\eqref{bounds:lipschitz} implies that for all $(x,p)\in\Omega\times \mathbb{R}^n$, the set $\partial_p H(x,p)$ is contained in the closed ball of radius $c_4=\lVert b \rVert_{C(\overline{\Omega}\times\mathcal{A};\mathbb{R}^n)}$ centred at the origin.

Given a function $v\in W^{1,1}(\Omega)$, we say that a real-valued vector field $\tilde{b}:\Omega\to\mathbb{R}^n$ is a \emph{measurable selection of $\partial_pH(\cdot,\nabla v)$} if $\tilde{b}$ is Lebesgue measurable and $\tilde{b}(x)\in \partial_p H(x,\nabla v(x))$ for a.e.\ $x\in \Omega.$
The uniform boundedness of the subdifferential sets implies that any measurable selection $\tilde{b}$ of $\partial_p H(\cdot,\nabla v)$ must belong to $L^\infty(\Omega;\mathbb{R}^n)$. 
Thus, the correspondence between a function $v\in W^{1,1}(\Omega)$ and the set of all measurable selections of $\partial_p H(\cdot,\nabla v)$ defines a set-valued map between $W^{1,1}(\Omega)$ and $L^\infty(\Omega;\mathbb{R}^n)$.

\begin{definition}\label{Def1}
	Let $H$ be the function given by \eqref{1''}. We define the set-valued map $D_pH\colon W^{1,1}(\Omega)\rightrightarrows L^{\infty}(\Omega;\mathbb{R}^n)$ by
	$$D_pH[v]\coloneqq \left\{\tilde{b}\in L^{\infty}(\Omega;\mathbb{R}^n):\tilde{b}(x)\in\partial_pH(x,\nabla v(x)) \text{ \emph{for a.e.}\ }x\in\Omega\right\}.$$
\end{definition}
We show in Lemma~\ref{Prop1} below that $D_pH[v]$ is nonempty for all $v$ in $ W^{1,1}(\Omega)$. 

\section{Continuous Problem and Main Results}\label{Sec3}
\subsection{Problem Statement}
We now introduce the notion of weak solution for the MFG PDI \eqref{sys}. 
\begin{definition}[Weak Solution of \eqref{sys}]\label{weakdef}
	We say that a pair $(u,m)\in H_0^1(\Omega)\times H_0^1(\Omega)$ is a weak solution of \eqref{sys} if there exists a vector field 
	$\tilde{b}_*\in D_pH[u]$ such that, for all $\psi,\phi\in H_0^1(\Omega)$, there hold
	\begin{subequations}\label{weakform}
		\begin{align}
			\int_{\Omega}\nu\nabla u\cdot\nabla \psi+H(x,\nabla u)\psi + \kappa u\psi\,\mathrm{d}x &= \langle F[m],\psi\rangle_{H^{-1}\times H_0^1}, \label{weakform1}\\ 
			\int_{\Omega}\nu\nabla m\cdot\nabla \phi+m\tilde{b}_*\cdot\nabla \phi+\kappa m\phi \,\mathrm{d}x &=\langle G,\phi\rangle_{H^{-1}\times H_0^1}.\label{weakform2} 
		\end{align} 
	\end{subequations}
\end{definition}

The weak formulation of the problem given in Definition \ref{weakdef} can be reformulated in terms of a PDI.
In particular, recalling the definition of the set-valued map $D_pH$ in Definition \ref{Def1} above, for given $m,u\in H_0^1(\Omega)$, let
\begin{multline}\label{H-1set}
	\mathrm{div}\left(mD_pH[u]\right)\coloneqq
	\\
	\left\{g\in H^{-1}(\Omega):\exists\tilde{b}\in D_pH[u]\text{ s.t.\ }\langle g,\phi\rangle_{H^{-1}\times H_0^1}=-\int_{\Omega}m\tilde{b}{\cdot}\nabla \phi\,\mathrm{d}x\;\forall \phi\in H_0^1(\Omega)\right\}.
\end{multline}
In other words, the set $\mathrm{div}\left(m D_pH[u]\right)$ is the set of all distributions in $H^{-1}(\Omega)$ of the form $\mathrm{div}(m \tilde{b})$ where $\tilde{b}\in D_pH[u]$.
Then, the definition of a weak solution in Definition \ref{weakdef} is equivalent to requiring that $(u,m)\in H_0^1(\Omega)\times H_0^1(\Omega)$ solves the following pair of conditions which hold in the sense of distributions in $H^{-1}(\Omega)$:
\begin{subequations}\label{eq:pdi_distributional}
	\begin{align}
		&-\nu\Delta u + H(x,\nabla u)+\kappa u = F[m] 
		,
		\\
		&-\nu\Delta m +\kappa m - G(x) \in \text{div}\left(mD_pH[u]\right).
	\end{align}
\end{subequations}
Therefore, the PDI~system~\eqref{eq:pdi_distributional} is the weak formulation of~\eqref{sys}.

\begin{remark}\label{rem:ducasse}
	\emph{In~\cite[Definition~3.1]{ducasse2020second}, Ducasse \emph{et al.} propose a definition of weak solutions for problems with Hamiltonians of the form $\overline{H}(\overline{m},p)\coloneqq \overline{\kappa} (\overline{m})|p|$. In particular, their definition for a weak solution $(\overline{u},\overline{m})$ involves an advective velocity term $V$ in the KFP equation replaced by a possibly nonunique vector field $V$ that satisfies the conditions (in the present notation)
	\begin{equation}\label{duc1}
		\lVert V\rVert_{L^{\infty}(\Omega;\mathbb{R}^n)}\leq \overline{\kappa}(\overline{m}),\quad V(x)\cdot \nabla \overline{u}(x)= \overline{\kappa}(\overline{m}) |\nabla \overline{u}(x)|\text{ for a.e.\ }x\in\Omega.
	\end{equation}	
	Although it is not stated therein, it is straightforward to check that the conditions in~\eqref{duc1} are equivalent to requiring that $V$ belongs to the partial subdifferential $\partial_p \overline{H}(\overline{m},\nabla \overline{u})$. Thus, modulo the dependence of the Hamiltonian on the density of players, our approach significantly generalizes that of~\cite{ducasse2020second} to more general nondifferentiable Hamiltonians.}
\end{remark}


\subsection{Main Results}
The first main result for the continuous problem \eqref{weakform} is the following.
\begin{theorem}[Existence of Weak Solutions]\label{thm1-existence}
	There exists a pair $(u,m)\in H_0^1(\Omega)\times H_0^1(\Omega)$ that is a weak solution of \eqref{sys} in the sense of Definition \ref{weakdef} satisfying
	\begin{gather}
		\|m\|_{H^1(\Omega)}\leq C^*\|G\|_{H^{-1}(\Omega)},\label{continuous-estimates-m}
		\\
		\|u\|_{H^1(\Omega)}\leq C^{**}\left(\|G\|_{H^{-1}(\Omega)}+\|f\|_{C(\overline{\Omega}\times\mathcal{A})}+1\right), \label{continuous-estimates-u}
	\end{gather}
	for some constants $C^*,C^{**}\geq 0$ depending only on $n$, $\Omega$, $\nu$, $\|b\|_{C(\overline{\Omega}\times\mathcal{A};\mathbb{R}^n)}$, $\kappa$ and $c_1$. 
\end{theorem}
The second main result ensures uniqueness of weak solutions of \eqref{sys} under a monotonicity condition on $F$ that is similar to the one that was used by Lasry and Lions in \cite{lasry2007mean}.
Since we also consider problems with source terms, we shall further require nonnegativity of $G$ in the sense of distributions.
\begin{theorem}[Uniqueness of Weak Solutions]\label{thm2-uniqueness} 
	If $F$ is strictly monotone and $G$ is nonnegative in the sense of distributions in $H^{-1}(\Omega)$, then there exists a unique weak solution pair $(u,m)\in H^1_0(\Omega)\times H^1_0(\Omega)$ to \eqref{sys} in the sense of Definition \ref{weakdef}.
\end{theorem}

\begin{remark}
	\emph{To avoid any confusion, we stress that Theorem~\ref{thm2-uniqueness} guarantees the uniqueness of the weak solution pair $(u,m)$ under the relevant hypotheses, although the advective vector field $\tilde{b}\in D_pH[u]$ that appears in Definition \ref{weakdef} may be nonunique.
	Note also that the monotonicity condition on $F$ is similar to the monotonicity condition on the coupling term used by Lasry \& Lions in~\cite{lasry2006jeux1} for classical solutions to ergodic mean field game systems with $C^1$ Hamiltonians.}
\end{remark}

\subsection{An Example}\label{sec:example}

Let us consider an example problem that motivates the definition of Definition \ref{Def1} and illustrates some challenges that arise in the case of nondifferentiable Hamiltonians.
\begin{eg}\label{eg-1}
	\emph{For simplicity, we consider a system in one space dimension
	\begin{equation}\label{sys-1d-example}
		-  u_{xx}+H(u_x)=F[m], \qquad -  m_{xx} -G(x)\in \left(m\partial_p H( u_x)\right)_x\qquad\text{ in }\Omega,
	\end{equation}
	where $\Omega = (-1,1)\subset  \mathbb{R}$, along with the homogeneous Dirichlet boundary conditions $u=m=0$ on $\partial\Omega=\{-1,1\}$.
	We consider a MFG where the control set of the players is $\mathcal{A}=\{-1,1\}$, where the drift $b$ and running-cost component $f$ are given by $b(x,\alpha)=\alpha$ and $f(x,\alpha)=0$ for all $\alpha \in \mathcal{A}$ and $x\in\Omega$.
	The resulting Hamiltonian in~\eqref{1''} therefore simplifies to $H(u_x)=\sup_{\alpha\in\mathcal{A}}(\alpha u_x)=\abs{u_x}$.
	Let $G(x) = \chi_{[-1/2,1/2]}$ where $\chi_{[-1/2,1/2]}$ denotes the indicator function for the interval $[-1/2,1/2]$, and let the coupling term $F[z]\coloneqq z-h+1-\chi_{[-1/2,1/2]}$ for all $z\in L^2(\Omega)$, where the function $h$ is defined by
	\begin{equation}
		h(x)\coloneqq 
		\begin{cases}
			\frac{1}{2} \left(1-\mathrm{e}^{\abs{x}-1} \right)&\text{ if }x\in [-1,-1/2]\cup [1/2,1],
			\\
			\frac{1}{2} \left(\frac{5}{4}-\mathrm{e}^{-\frac{1}{2}} -x^2 \right) &\text{ if }x\in [-1/2,1/2].
		\end{cases}
	\end{equation}
	Note that $G$ is nonnegative and $F$ is strongly monotone on $L^2(\Omega)$, so~\eqref{sys-1d-example} admits a unique solution.
	The problem can be solved analytically, and the exact solution is
	\begin{equation}
		m(x) = h(x), \quad u(x) =
		\begin{cases}
			1-\abs{x} + \mathrm{e}^{-\frac{1}{2}}-\mathrm{e}^{\frac{1}{2}-\abs{x}} &\text{ if }x\in [-1,-1/2]\cup [1/2,1], \\
			\mathrm{e}^{-\frac{1}{2}} - \frac{1}{2} &\text{ if }x\in [-1/2,1/2].
		\end{cases}
	\end{equation}
	Furthermore, we find that the unique function $\tilde{b}_*\in D_pH[u]$ for which $-m_{xx}-(\tilde{b}_* m)_x = G$ in $\Omega$ is given by $\tilde{b}_*|_{[-1,-1/2]}=1$, $\tilde{b}_*|_{(-1/2,1/2)}=0$, and $\tilde{b}_*|_{[1/2,1]}=-1$.
	To see that $\tilde{b}_*$ is unique, note that if $\hat{b}_*\in L^\infty(\Omega)$ also satisfies $-m_{xx}-(\hat{b}_* m)_x =G$ in $\Omega$, then $m(\tilde{b}_*-\hat{b}_*)$ is constant in $(-1,1)$. Since $m$ satisfies the homogeneous Dirichlet condition on the boundary we deduce that the constant must be zero, and since $m$ is nonvanishing inside $\Omega$, we find that $\tilde{b}_*=\hat{b}_*$ a.e.\ in $\Omega$, thus $\tilde{b}_*$ is unique.}
	
	\emph{This example illustrates several points. First, the solution $m$ is not continuously differentiable in the interior of the domain and thus $m\notin H^2(\Omega)$, despite the facts that $F[m]$ and $G$ are in $L^2(\Omega)$. This is due to the jumps in the vector field $b_*$ at $x=\pm 1/2$. This example shows how loss of smoothness of the solution can occur in the interior of the domain for problems with nondifferentiable Hamiltonians.}
	
	\emph{The second point concerns the motivation for choosing the subdifferential $D_pH[u]$ as the appropriate set for defining the possible advective fields in Definition \ref{Def1}.
	Observe that in the region $(-1/2,1/2)$, the set of optimal feedback controls is the whole control set $\mathcal{A} = \{-1,1\}$ since $u_x =0$ in $(-1/2,1/2)$.
	However $\tilde{b}_*(x) \notin \mathcal{A}$ for all $x\in(-1/2,1/2)$, i.e.\ $\tilde{b}_*$ does not coincide with any optimal feedback policy in this region.
	Since $\tilde{b}_*$ is unique in this example, it is then clear that in Definition \ref{Def1}, we cannot generally require that $\tilde{b}_*$ necessarily belong to smaller sets than $D_pH[u]$,  such as the set of drifts generated by optimal feedback policies.}
\end{eg}

\section{Analysis of the Continuous Problem}\label{Sec 4}
\subsection{Preliminary Results}\label{Sec_4_prelim}
We begin by introducing the point-wise maximizing set of the Hamiltonian.
Define the set-valued map $\Lambda\colon\Omega\times\mathbb{R}^n\rightrightarrows \mathcal{A}$ by
\begin{equation}
	\Lambda(x,p)\coloneqq\text{argmax}_{\alpha\in\mathcal{A}}\{b(x,\alpha)\cdot p-f(x,\alpha)\} \quad \forall (x,p)\in \Omega\times \mathbb{R}^n,
\end{equation}
Note that $\Lambda(x,p)$ is nonempty for all $x\in \Omega$ and all $p\in\mathbb{R}^n$ since $\mathcal{A}$ is compact and the functions $b$ and $f$ are uniformly continuous.
The following Lemma, which is a consequence of \cite[Proposition 4.4]{aubin1993optima}, shows the link between the sets of maximizing controls and the subdifferentials of the Hamiltonian.
\begin{lemma}\label{convv}
	Let $H$ be given by \eqref{1''}. Then 
	\begin{equation}\label{convv-formula}
		\partial_pH(x,p)=\conv\left\{b(x,\alpha):\alpha\in \Lambda(x,p)\right\}\quad\forall (x,p)\in{\Omega}\times\mathbb{R}^n.
	\end{equation}
\end{lemma}

\begin{remark}\label{rem:drift_meaning}
	\emph{Lemma \ref{convv} offers some insight into the significance of the term $\tilde{b}_*$ appearing in Definition \ref{Def1} from a modelling perspective.
	Indeed, it shows that for a.e.\ $x\in\Omega$, $\tilde{b}_*(x)$ is in the closed convex hull of the set of drifts generated by the optimal controls from $\Lambda(x,\nabla u(x))$. 
	This suggests that the players in the same region of state space can make distinct choices among nonunique optimal controls, leading to an aggregate advective flux for the player density.}
\end{remark}

To show that $D_pH$ possess nonempty images, we introduce an auxiliary set-valued map.  For a given $v\in W^{1,1}(\Omega)$, let $\Lambda[v]$ denote the set of all Lebesgue measurable functions $\alpha^*:\Omega\to\mathcal{A}$ that satisfy $\alpha^*(x)\in\Lambda(x,\nabla v(x))$ for a.e.\ $x\in\Omega.$ We will refer to each element of $\Lambda[v]$ as \emph{a measurable selection of }$\Lambda(\cdot,\nabla v(\cdot))$. 
It is known that $\Lambda[v]$ is nonempty for each $v\in W^{1,1}(\Omega)$, see e.g.\ \cite[Theorem~10]{smears2014discontinuous}, where the proof the existence of measurable selections ultimately rests upon the Kuratowski and Ryll--Nardzewski Selection Theorem~\cite{kuratowski1965general}.
We now show that the set-valued map $D_pH$ has nonempty images.
\begin{lemma}\label{Prop1}
	For each $v\in W^{1,1}(\Omega)$, the set $D_pH[v]$ is a nonempty subset of $L^{\infty}(\Omega;\mathbb{R}^n)$, and we have the uniform bound
	\begin{equation}\label{Hsubdiff-bound}
		\sup_{v\in W^{1,1}(\Omega)}\left[\sup_{\tilde{b}\in D_pH[v]}\|\tilde{b}\|_{L^{\infty}(\Omega;\mathbb{R}^n)}\right]\leq \|b\|_{C(\overline{\Omega}\times\mathcal{A};\mathbb{R}^n)}.
	\end{equation}
\end{lemma}
\begin{proof}
	Let $v\in W^{1,1}(\Omega)$ be given. We need to show  $D_pH[v]$ is nonempty. nonemptiness of $\Lambda[v]$ implies that there exists a Lebesgue measurable map $\alpha^*:\Omega\to\mathcal{A}$ such that $\alpha^*(x) \in \Lambda(x, \nabla v(x))$ for a.e.\ $x\in \Omega$, and thus $H(x,\nabla v(x))=b(x,\alpha^*(x))\cdot\nabla v(x)-f(x,\alpha^*(x))$
	for a.e.\ $x\in\Omega$. Now, suppose $q\in\mathbb{R}^n$ is arbitrary. We find by definition of $H(x,q)$ that, for a.e.\ $x\in\Omega$, 
	\begin{equation*}
		\begin{split}
			H(x,q)&\geq b(x,{\alpha^*(x)})\cdot q-f(x,\alpha^*(x))
			\\
			&=b(x,{\alpha^*(x)})\cdot q+H(x,\nabla v(x))-b(x,{\alpha^*(x)})\cdot\nabla v(x)
			\\
			&=H(x,\nabla v(x))+b(x,{\alpha^*(x)})\cdot\left(q-\nabla v(x)\right).
		\end{split}
	\end{equation*}
	It follows then that $ b(x,\alpha^*(x))\in \partial_p H(x,\nabla v(x))$ for a.e.\ $x\in\Omega$.
	Furthermore, we have $b(\cdot,\alpha^*(\cdot)) \in L^\infty(\Omega;\mathbb{R}^n)$ since $\|b(\cdot,{\alpha^*(\cdot)})\|_{L^{\infty}(\Omega;\mathbb{R}^n)}\leq \|b\|_{C(\overline{\Omega}\times\mathcal{A};\mathbb{R}^n)}.$ Hence, $D_pH[v]$ is nonempty, as claimed.
	Finally, the bound~\eqref{Hsubdiff-bound} follows immediately from the fact that, for all $(x,p)\in\Omega\times \mathbb{R}^n$, the subdifferential set $\partial_p H(x,p)$ is contained in the closed ball of radius $\lVert b \rVert_{C(\overline{\Omega}\times\mathcal{A};\mathbb{R}^n)}$ centred at the origin.
\end{proof} 

The following Lemma shows that $D_pH$ has a certain closure property with respect to convergent sequences of its arguments and their measurable selections.
\begin{lemma}\label{inclusion}
	Suppose $\{v_j\}_{j\in\mathbb{N}}\subset H^1(\Omega)$, $\{\tilde{b}_{j}\}_{j\in\mathbb{N}}\subset L^{\infty}(\Omega;\mathbb{R}^n) $ are sequences such that $\tilde{b}_j\in D_pH[v_j]$ for all $j\in\mathbb{N}$. If $v_j\to v$ in $H^1(\Omega)$ and $\tilde{b}_j\rightharpoonup^* \tilde{b}$ in $L^{\infty}(\Omega;\mathbb{R}^n)$ as $j\to \infty$, then $\tilde{b}\in D_pH[v]$. 
\end{lemma}
\begin{proof}	
	Introduce the set $Y\coloneqq\left\{v\in L^2(\Omega):v(x)\geq 0\text{ for a.e. }x\in\Omega\right\}$ of nonnegative a.e.\ functions in $L^2(\Omega)$, and note that Mazur's Theorem implies that $Y$ is weakly closed in $L^2(\Omega)$ since it is convex and strongly closed.
	Let $q\in\mathbb{R}^n$ be a fixed but arbitrary vector.
	Define the sequence of real-valued functions $\{\omega_j\}_{j=1}^{\infty}\subset L^2(\Omega)$ by 
	\begin{equation}
		\omega_j(x)\coloneqq H(x,q)-H(x,\nabla v_j(x))-\tilde{b}_j(x)\cdot\left(q-\nabla v_j(x)\right)
	\end{equation}
	for each $j\in \mathbb{N}$ and a.e.\ $x \in \Omega$.
	It follows from the definitions of the subdifferential sets~\eqref{subdifferential} and Definition~\ref{Def1} that $\omega_j \in Y$ for each $j\geq 1$. The hypothesis of strong convergence of $\{v_j\}_{j\in\mathbb{N}}$ and weak-$*$ convergence of $\{\tilde{b}_j\}_{j\in\mathbb{N}}$ implies the weak convergence $\omega_j\rightharpoonup \omega$ in $L^2(\Omega)$ where
	\begin{equation}
		\omega(x)\coloneqq H(x,q)-H(x,\nabla v(x))-\tilde{b}(x)\cdot\left(q-\nabla v(x)\right)
	\end{equation}
	for a.e.\ $x\in \Omega$.
	Since $Y$ is weakly closed, it follows that $\omega \in Y$. Since $q \in \mathbb{R}^n$ is arbitrary and since $\mathbb{R}^n$ is separable, we conclude that $\tilde{b} \in D_pH[v]$.
\end{proof}

\subsection{Existence of Weak Solutions}
In this section, we prove Theorem \ref{thm1-existence}. To begin, we introduce  notation describing a collection of linear differential operators in weak form. Given $C_0\geq 0$, let $\mathcal{G}(C_0)$ denote the set of all operators $L:H_0^1(\Omega)\to H^{-1}(\Omega)$ of the form
\begin{equation}\label{eq:wmp_1}
	\langle Lu,v\rangle_{H^{-1}\times H_0^1}=\int_{\Omega}\nu\nabla u\cdot\nabla v+\tilde{b}\cdot\nabla u v+cuv\,\mathrm{d}x, 
\end{equation} 
where the coefficients satisfy
\begin{equation}\label{eq:wmp_2}  
	\lVert \tilde{b}\rVert_{L^{\infty}(\Omega;\mathbb{R}^n)}+\lVert c\rVert_{L^{\infty}(\Omega)}\leq C_0,\text{ and }  c\geq 0\text{ a.e.\ in } \Omega. 
\end{equation}
Moreover, given an operator $L\in \mathcal{G}(C_0)$ for some $C_0\geq 0$, we define $L^*:H_0^1(\Omega)\to H^{-1}(\Omega)$, the formal adjoint of $L$, by $\langle L^*w,v\rangle_{H^{-1}\times H_0^1} \coloneqq \langle Lv, w\rangle_{H^{-1}\times H_0^1}$ for all $ w,v\in H_0^1(\Omega).$ 
The invertibility of operators $L$ and their adjoints $L^*$ from the class $\mathcal{G}(C_0)$ is well-known, and follows from the the Fredholm Alternative together with the Weak Maximum Principle and the Comparison Principle (see \cite[Ch. 8, Ch. 10]{gilbarg2015elliptic}). 
In the analysis below, we will use the following stronger result, which shows that for fixed $C_0$, there is a uniform bound on the norm of the inverses of all operators and their adjoints from the class $\mathcal{G}(C_0)$.
\begin{lemma}\label{unif}
	Let $C_0\geq 0$ be given. For every operator $L\in\mathcal{G}(C_0)$, both $L$ and $L^*$ are boundedly invertible as mappings from $H^1_0(\Omega)$ to $H^{-1}(\Omega)$, and there exists a constant $C_1> 0$ depending on only $\Omega$, $n$, $\nu$, and  $C_0$ such that 
	\begin{equation}\label{eq:invertibility} 
		\sup_{L\in\mathcal{G}(C_0)}\max\left\{\left\|L^{-1}\right\|_{\mathcal{L}\left(H^{-1}(\Omega),H_0^1(\Omega)\right)},\left\|{L^*}^{-1}\right\|_{\mathcal{L}\left(H^{-1}(\Omega),H_0^1(\Omega)\right)}\right\}\leq C_1.
	\end{equation}
\end{lemma}

Moreover, we will use the following result that guarantees both well-posedness for a class of HJB equations in weak form and a useful continuity property. 
\begin{lemma}[Well-posedness of the HJB equation]\label{thm1}
	Let $m\in L^2(\Omega)$ be given. Then, there exists a unique $u\in H_0^1(\Omega)$ such that
	\begin{equation}\label{weaku}
		\int_{\Omega}\nu\nabla u\cdot\nabla \psi+ H(x,\nabla u)\psi+\kappa u\psi\,\mathrm{d}x = \langle F[m],\psi\rangle_{H^{-1}\times H_0^1} \text{ }\forall \psi\in H_0^{1}(\Omega).
	\end{equation}
	There exists a constant $C_2$ depending only on $\Omega$, $n$, $\nu$, $\kappa$, $\lVert b \rVert_{C(\overline{\Omega}\times \mathcal{A};\mathbb{R}^n)}$ and $c_1$ such that  
	\begin{equation}\label{eq:hjb_a_priori_bound}
		\lVert u \rVert_{H^1(\Omega)}\leq C_2\left(\lVert m \rVert_\Omega + \lVert f \rVert_{C(\overline{\Omega}\times \mathcal{A})}+1\right).
	\end{equation}
	Moreover, the solution $u$ depends continuously on $m$, i.e., if $\left\{m_j\right\}_{j\in\mathbb{N}}\subset L^2(\Omega)$ is such that $m_j\to m$ in $L^2(\Omega)$ as $j\to\infty$, then the corresponding sequence of solutions $\left\{u_j\right\}_{j\in\mathbb{N}}\subset H_0^1(\Omega)$ to the problem~\eqref{weaku} converges in $H_0^1(\Omega)$ to the unique solution $u$ of \eqref{weaku}.
\end{lemma}
The proofs of Lemmas~\ref{unif} and~\ref{thm1} are given in Appendix~\ref{appendix-a} for completeness. 
We can now show existence of weak solutions to \eqref{sys} in the sense of Definition \ref{weakdef} by an application of Kakutani's fixed-point theorem \cite[Ch.\ 9, Theorem 9.B]{ZMR0816732}. Our proof is similar to the proof of \cite[Theorem 3.3]{ducasse2020second}.

\begin{theorem}[Kakutani's fixed point theorem \cite{ZMR0816732}]\label{thm-kakutani}
    Suppose that 
    \begin{enumerate}
        \item $\mathcal{B}$ is a nonempty, compact, convex set in a locally convex space $X$;
        \item $\mathcal{V}:\mathcal{B}\rightrightarrows\mathcal{B}$ is a set-valued map such that $\mathcal{V}[\tilde{b}]$ is nonempty, closed and convex for all $\tilde{b}\in\mathcal{B}$; and
        \item $\mathcal{V}$ is upper semi-continuous.
    \end{enumerate}
    Then $\mathcal{V}$ has a fixed point: there exists $\tilde{b}_{*}\in\mathcal{B}$ such that $\tilde{b}_{*}\in \mathcal{V}[\tilde{b}_{*}]$. 
\end{theorem}
\begin{proof}[Proof of Theorem \ref{thm1-existence}]    
    {Recall that $c_4=\lVert b\rVert_{C(\overline{\Omega}\times\mathcal{A};\mathbb{R}^n)}$ is the Lipschitz constant of the Hamiltonian, c.f.~\eqref{bounds:lipschitz}.
    We equip the space $L^{\infty}(\Omega;\mathbb{R}^n)$ with its weak-$*$ topology, noting that it is then a locally convex topological space. Let $\mathcal{B}$ denote the ball
    $$\mathcal{B}\coloneqq \left\{\tilde{b}\in L^{\infty}(\Omega;\mathbb{R}^n): \|\tilde{b}\|_{L^{\infty}(\Omega;\mathbb{R}^n)}\leq c_4\right\}.$$ 
    We note that $\mathcal{B}$ is nonempty, convex, and also closed in the weak-$*$ topology.
     Since $L^1(\Omega;\mathbb{R}^n)$ is separable, the weak-$*$ topology on $\mathcal{B}$ is metrizable \cite[Ch.\ 15]{royden2018real}. Moreover, Helly's Theorem implies that $\mathcal{B}$ is a compact.

    Let $M: \mathcal{B}\to H_0^1(\Omega)$ be the map defined as follows: for each $\tilde{b}\in \mathcal{B}$, let $M[\tilde{b}]$ in $H_0^1(\Omega)$ be the unique solution of
    \begin{equation}\label{M-map-def}
        \int_{\Omega}\nu\nabla M[\tilde{b}]\cdot\nabla \phi+M[\tilde{b}]\tilde{b}\cdot\nabla \phi+\kappa M[\tilde{b}] \phi\,\mathrm{d}x=\langle G,\phi\rangle_{H^{-1}\times H_0^1}\text{  }\text{ }\forall\phi\in H_0^1(\Omega).
    \end{equation}
    The map $M$ is well-defined thanks to Lemma \ref{unif}. Next, let $U:L^2(\Omega)\to H_0^1(\Omega)$ be the map defined as follows: for each $m\in L^2(\Omega)$, let $U[m]\in H_0^1(\Omega)$ denote the unique solution of
    \begin{equation}\label{U-map-def}
        \int_{\Omega}\nu\nabla U[m]\cdot\nabla \psi+H(x,\nabla U[m])\psi +\kappa U[m] \psi\,\mathrm{d}x=\langle F[m],\psi\rangle_{H^{-1}\times H_0^1} \text{ }\text{ }\forall\psi\in H_0^1(\Omega).
    \end{equation}
    The map $U$ is well-defined by Lemma \ref{thm1}. 

    Now, define the set-valued map $\mathcal{V}:\mathcal{B}\rightrightarrows L^{\infty}(\Omega;\mathbb{R}^n)$ as follows: given $\tilde{b}\in \mathcal{B}$, let 
    $$\mathcal{V}[\tilde{b}]\coloneqq  D_pH\big[U\big[M\big[\tilde{b}\big]\big]\big].$$ Lemma \ref{Prop1} implies that $\mathcal{V}[\tilde{b}]\subset\mathcal{B}$ for each $\tilde{b}\in\mathcal{B}$, so $\mathcal{V}:\mathcal{B}\rightrightarrows \mathcal{B}$. Moreover, for every $\tilde{b}\in \mathcal{B}$, the set $\mathcal{V}[\tilde{b}]$ is nonempty and convex. Indeed, for each $\tilde{b}\in\mathcal{B}$ the set $\mathcal{V}[\tilde{b}]$ is nonempty by Lemma \ref{Prop1}. Also, $\mathcal{V}[\tilde{b}]$ is convex since $\partial_pH$ has convex images. Moreover, $\mathcal{V}[\tilde{b}]$ is closed for all $\tilde{b}\in\mathcal{B}$ thanks to Lemma~\ref{inclusion}. 
    
    The existence of a weak solution to the MFG PDI \eqref{sys} in the sense of Definition \ref{weakdef} is equivalent to showing the existence of a fixed point of $\mathcal{V}$, i.e.\ that there exists $\tilde{b}_*\in \mathcal{B}$ such that $\tilde{b}_*\in\mathcal{V}[\tilde{b}_*].$ Indeed, if $\tilde{b}_*\in \mathcal{B}$ satisfies $\tilde{b}_*\in\mathcal{V}[\tilde{b}_*]$ then a solution pair $(u,m)$ of the weak formulation \eqref{weakform} is given by $m\coloneqq M[\tilde{b}_*]$ and $u\coloneqq U[m]$ with $\tilde{b}_*\in D_pH[u]$, while the converse is obvious.
    
    We now verify that $\mathcal{V}$ is upper-semicontinuous. To this end, it suffices to prove that the graph of $\mathcal{V}$ is closed; c.f.\ \cite[Ch.\ 1, Corollary 1, p.\ 42]{MR0755330}.  Let $\mathcal{W}$ denote the graph of $\mathcal{V}$, which is defined by 
    \begin{equation}\label{graph-def}
        \mathcal{W}\coloneqq \left\{(\tilde{b},\overline{b})\in \mathcal{B}\times\mathcal{B}:\overline{b}\in\mathcal{V}[\tilde{b}]\right\}.
    \end{equation}
    Since $\mathcal{B}$ is metrizable, to show that the graph $\mathcal{W}$ is a closed it is enough to show that whenever a sequence $\{(\tilde{b}_i,\overline{b}_i)\}_{i\in\mathbb{N}}\subset \mathcal{W}$ converges weakly-$*$ in $\mathcal{B}\times\mathcal{B}$ to a point $(\tilde{b},\overline{b})$ as $i\to\infty$, then $(\tilde{b},\overline{b})\in\mathcal{W}$,  which is to say $\overline{b}\in \mathcal{V}[\tilde{b}]$. Let us then suppose that we are given a sequence $\{(\tilde{b}_i,\overline{b}_i)\}_{i\in\mathbb{N}}\subset \mathcal{W}$ converges weakly-$*$ in $\mathcal{B}\times\mathcal{B}$ to a point $(\tilde{b},\overline{b})$ as $i\to\infty$. To begin, we claim that $M[\tilde{b}_i]\to M[\tilde{b}]$ in $L^2(\Omega)$ as $i\to\infty$. Indeed, since $\{\tilde{b}_i\}_{i\in\mathbb{N}}\subset\mathcal{B}$, for each $i\in\mathbb{N}$ we apply Lemma \ref{unif} to obtain the uniform bound
    \begin{equation}\label{M_n}
        \sup_{i\in\mathbb{N}}\|M[\tilde{b}_i]\|_{H^1(\Omega)}\leq C_1\|G\|_{H^{-1}(\Omega)}.
    \end{equation}
    We deduce from~\eqref{M_n} that any given subsequence $\{M[\tilde{b}_{i_j}]\}_{j\in\mathbb{N}}$ is uniformly bounded in $H_0^1(\Omega)$. The Rellich--Kondrachov compactness theorem then implies that there exists a further subsequence $\{M[\tilde{b}_{i_{j_s}}]\}_{s\in\mathbb{N}}$ and $m\in H_0^1(\Omega)$ such that $M[\tilde{b}_{i_{j_s}}]\rightharpoonup m$ in $H_0^1(\Omega)$ and $M[\tilde{b}_{i_{j_s}}]\to m$ in $L^2(\Omega)$ as $s\to\infty$. By $L^{\infty}$-weak-$*$ $\times$ $L^2$-strong convergence, we also have that $M[\tilde{b}_{i_{j_s}}]\tilde{b}_{i_{j_s}}\rightharpoonup m\tilde{b}$ in $L^2(\Omega;\mathbb{R}^n)$ as $s\to\infty$. Passing to the limit in the KFP equation \eqref{M-map-def} satisfied by $M[\tilde{b}_{i_{j_s}}]$ for $s\in\mathbb{N}$, we deduce that $m$ satisfies 
    \begin{equation}\label{m-alt-eqn}
        \int_{\Omega}\nu\nabla m\cdot\nabla \phi+m\tilde{b}\cdot\nabla \phi+\kappa m \phi\,\mathrm{d}x=\langle G,\phi\rangle_{H^{-1}\times H_0^1}\text{  }\text{ }\forall\phi\in H_0^1(\Omega).
    \end{equation}
    But by definition of $M[\tilde{b}]$ in \eqref{M-map-def}, we obtain necessarily $m=M[\tilde{b}]$ in $H_0^1(\Omega)$. 
    In summary, every subsequence of $\{M[\tilde{b}_i]\}_{i\in \mathbb{N}}$ has a further subsequence that converges to $M[\tilde{b}]$ in $L^2(\Omega)$ as $i\to \infty$.
    It follows that the entire sequence $\{M[\tilde{b}_i]\}_{i\in\mathbb{N}}$ satisfies $M[\tilde{b}_i]\to M[\tilde{b}]$ in $L^2(\Omega)$ as $i\to\infty$. Lemma \ref{thm1} then implies that $U[M[\tilde{b}_i]]\to U[M[\tilde{b}]]$ in $H_0^1(\Omega)$ as $i\to\infty$. By hypothesis, $\overline{b}_i\in \mathcal{V}[\tilde{b}_i]=D_pH[U[M[\tilde{b}_i]]]$ for $i\in\mathbb{N}$ and $\overline{b}_i\rightharpoonup^* \overline{b}$ in $L^{\infty}(\Omega;\mathbb{R}^n)$ as $i\to\infty$. We conclude from Lemma~\ref{inclusion} that $\overline{b}\in D_pH[U[M[\tilde{b}]]]$, i.e.\ $\overline{b}\in \mathcal{V}[\tilde{b}]$. We have therefore shown that $\mathcal{W}$ is closed, so $\mathcal{V}$ is upper semi-continuous.
    
    We have thus shown that the map $\mathcal{V}:\mathcal{B}\rightrightarrows\mathcal{B}$ satisfies the conditions of Kakutani's fixed-point theorem, so $\mathcal{V}$ admits a fixed point and therefore there exists a weak solution to \eqref{sys} in the sense of Definition \ref{weakdef}.} 

    The bound \eqref{continuous-estimates-m} follows directly from the KFP equation satisfied by the density function $m$ in the weak formulation \eqref{weakform} and an application of Lemma \ref{unif}. The bound \eqref{continuous-estimates-u} then follows from \eqref{continuous-estimates-m} and \eqref{eq:hjb_a_priori_bound}.
\end{proof}

\subsection{Uniqueness of Weak Solutions} 
Using Definition \ref{Def1}, in addition to a strict monotonicity condition on $F$ and the nonnegativity of $G$ in the sense of distributions in $H^{-1}(\Omega)$, we obtain uniqueness of weak solutions in the sense of Definition \ref{weakdef} through a monotonicity argument similar to~\cite{lasry2007mean}.
\begin{proof}[Proof of Theorem \ref{thm2-uniqueness}]
	Suppose that there exist $(u_i,m_i)$, $i\in\{1,2\}$, that each satisfy \eqref{weakform} with
	\begin{subequations}\label{eq:uniqueness_1}
		\begin{align}
			\int_{\Omega}\nu\nabla u_i\cdot\nabla \psi +H(x,\nabla u_i)\psi +\kappa u_i \psi\, \mathrm{d}x &=\langle F[m_i],\psi\rangle_{H^{-1}\times H_0^1} &&\forall \psi\in H^1_0(\Omega), \label{eq:uniqueness_1a}\\
			\int_{\Omega}\nu\nabla m_i\cdot\nabla \phi+m_i\tilde{b}_{i}\cdot\nabla \phi  +\kappa m_i\phi \,\mathrm{d}x &=\langle G,\phi\rangle_{H^{-1}\times H_0^1} &&\forall \phi\in H_0^1(\Omega),\label{eq:uniqueness_1b}
		\end{align}
	\end{subequations}
	for some $\tilde{b}_i\in D_p H[u_i] $.
	Since  $G$ is nonnegative in the sense of distributions in $H^{-1}(\Omega)$, the Comparison Principle (see the proof of~\cite[Theorem~10.7]{gilbarg2015elliptic}) applied to~\eqref{eq:uniqueness_1b} implies that $m_i\geq 0 $ a.e.\ in $\Omega$ for each $i\in\{1,2\}$.
	After choosing as test functions $\psi=m_1-m_2$ and $\phi=u_1-u_2$ in~\eqref{eq:uniqueness_1}, and subtracting the equations, we eventually find that
	\begin{equation}\label{7'} 
		\int_{\Omega}m_1\lambda_{12}+m_2\lambda_{21}\,\mathrm{d}x= \langle F[m_1]-F[m_2],m_1-m_2\rangle_{H^{-1}\times H_0^1},
	\end{equation}
	where the functions $\lambda_{ij}$ are defined by
	\begin{equation}
		\lambda_{ij}:=H(\cdot,\nabla u_i) - H(\cdot,\nabla u_j)+ \tilde{b}_{i}\cdot\nabla (u_j-u_i),\quad i,\,j\in\{1,2\}.
	\end{equation}
	By definition of $D_pH[u_i]$, in particular that $\tilde{b}_i (x) \in \partial_p H(x,\nabla u_i(x))$ for a.e.\ $x\in \Omega$, we see that $\lambda_{ij}\leq 0$ a.e.\ in $\Omega$ for $i,j\in\{1,2\}$. 
	Therefore, \eqref{7'} implies that $\langle F[m_1]-F[m_2],m_1-m_2\rangle_{H^{-1}\times H_0^1}\leq 0$, and thus the strict monotonicity condition~\eqref{eq:F_monotone} on $F$ implies that $m_1=m_2$.
	Consequently $u_1$ and $u_2$ satisfy \eqref{eq:uniqueness_1a} with identical right-hand side $F[m_1]=F[m_2]$.
	Therefore, Lemma \ref{thm1} implies that $u_1=u_2$.
	This shows that there is at most one weak solution to \eqref{sys} in the sense of Definition \ref{weakdef}.
\end{proof}

\begin{remark}
	\emph{In cases where the MFG PDI \eqref{weakform} admits a unique solution $(u,m)$, the collection of transport vectors $\tilde{b}_*\in D_pH[u]$ for which \eqref{weakform} holds constitutes an equivalence class of vector fields under the following equivalence relation: for vector fields $\tilde{b}_1,\tilde{b}_{2}\in D_pH[u]$, $$\tilde{b}_1\sim \tilde{b}_{2} \quad \text{if and only if} \quad\int_{\Omega}m\tilde{b}_1\cdot\nabla \phi\,\mathrm{d}x=\int_{\Omega}m\tilde{b}_{2}\cdot\nabla \phi\,\mathrm{d}x\quad\forall\phi\in H_0^1(\Omega).$$ Therefore, whenever $\tilde{b}_1,\tilde{b}_2$ are in the above equivalence class, $\text{div}(m\tilde{b}_1-m\tilde{b}_2)=0$ in the sense of distributions in $H^{-1}(\Omega)$. For instance, in three space dimensions, this implies that the vector $m\tilde{b}_1-m\tilde{b}_2$ is the curl of a vector potential in a distributional sense.}
\end{remark}

\section{Monotone Continuous Galerkin Finite Element Scheme}\label{Sec5}

In this section we introduce a monotone finite element scheme for approximating solutions to the weak formulation \eqref{weakform}. In the sequel,  we shall further assume that $\Omega$ is a polyhedron, in addition to the earlier assumption that it is a bounded connected open set with Lipschitz boundary.

\subsection{Notation}

A mesh $\mathcal{T}$ is a collection of closed $n$-dimensional simplices, called elements, $K$ with nonoverlapping interiors that satisfy
$\overline{\Omega}=\bigcup_{K\in\mathcal{T}}K.$ Each vertex of an element $K$ of a given mesh is called a node of the element. A face of an element $K\in\mathcal{T}$ is the convex hull of a collection of $n$ nodes of $K$ which has positive $(n-1)$-dimensional Hausdorff measure (c.f.  \cite{di2011mathematical}). For instance, when the space dimension $n=2$, each face of an element $K\in\mathcal{T}$ is one of its three edges. 
We will always assume that a given mesh $\mathcal{T}$ is conforming (or often called matching) \cite{di2011mathematical}, i.e.\ for any element $K\in\mathcal{T}$ with nodes $\{x_0,\cdots,x_n\}$, the set $\partial K\cap \partial K'$ for each element $K'\in\mathcal{T}$, $K'\neq K$, is the convex hull of a (possibly empty) subset of $\{x_0,\cdots,x_n\}$. 
Let $\{\mathcal{T}_k\}_{k\in\mathbb{N}}$ be a given sequence of conforming meshes. For each $k\in\mathbb{N}$, let the mesh-size of a given mesh $\mathcal{T}_k$ be defined by $h_k\coloneqq \max_{K\in\mathcal{T}_k}{\text{diam}(K)}$. We assume that $h_k\to 0$ as $k\to \infty$. We assume that $\{\mathcal{T}_k\}_{k\in\mathbb{N}}$ is shape-regular, i.e.\ there exists a real-number $\delta>1$, independent of $k\in\mathbb{N}$, such that $\forall k\in\mathbb{N},\,\forall K\in\mathcal{T}_k,\text{ }\text{diam}(K)\leq \delta \rho_{K},$ where $\rho_K$ denotes the radius of the largest inscribed ball in the element $K$. We assume in addition that the family of meshes $\{\mathcal{T}_k\}_{k\in\mathbb{N}}$ is nested, i.e.\ for each $k\in\mathbb{N}$ the mesh $\mathcal{T}_{k+1}$ is obtained from $\mathcal{T}_k$ via an admissible subdivision of each element of $\mathcal{T}_k$ into simplices.

Given an element $K$, we let $\mathcal{P}_1(K)$ denote the vector space of $n$-variate real-valued polynomials of total degree 1 that are defined on $K$. 
The discretization of the continuous problem \eqref{weakform} is based on the following finite element spaces: 
$$V_k\coloneqq \{v\in H_0^1(\Omega): v|_K\in\mathcal{P}_1(K)\text{ }\forall K\in \mathcal{T}_k\}\quad\forall k\in\mathbb{N}.$$
Given $k\in\mathbb{N}$, the space $V_k$ admits a unique nodal basis of hat functions that we denote by $\{\xi_1,\cdots,\xi_{N_k}\}$, which corresponds to a maximal collection of nodes $\{x_1,\cdots,x_{N_k}\}$ of the mesh $\mathcal{T}_k$, such that $\xi_i (x_j)=\delta_{ij}$ for $i,\,j\in\{1,\cdots,N_k\}$, where $\delta_{ij}$ is the Kronecker delta.
Moreover, $V_k$ inherits the standard norm on $H_0^1(\Omega)$ and we denote this norm by $\|\phi\|_{V_k}\coloneqq \|\phi\|_{H^1(\Omega)}$ for $\phi\in V_k$. Note that, due to nestedness of the sequence of meshes $\{\mathcal{T}_k\}_{k\in\mathbb{N}}$, $V_k$ is a closed subspace of $V_{k+1}$ for each $k\in\mathbb{N}$. In addition, the union $\bigcup_{k\in\mathbb{N}}V_k$ is dense in $H_0^1(\Omega)$. We let $V_k^*$ denote the space of continuous linear functionals on $V_k$ with standard norm denoted by $\|\cdot\|_{V_k^*}$.  For any operator $\mathcal{Z}:V_k\to V_k^*$ we define the adjoint operator $\mathcal{Z}^*:V_k\to V_k^*$ by 
$\langle \mathcal{Z}^*w, v\rangle_{V_k^*\times V_k}\coloneqq\langle \mathcal{Z} v, w\rangle_{V_k^*\times V_k}$ for all $w,v\in V_{k}$. 

Let $k\in\mathbb{N}$ be given. For $K\in\mathcal{T}_k$, we denote by $\{\psi_{k,0}^K,\cdots, \psi_{k,n}^{K}\}\subset V_{k}$ the set of nodal basis functions associated with the $n+1$ nodes of $K$ and let
\begin{equation}\label{sigmaK}
	\sigma_{K}^{k}\coloneqq \text{diam}(K)\min_{0\leq i\leq n}\left|\nabla \psi_{k,i}^K\right|, \quad \sigma^{k}\coloneqq \min_{K\in\mathcal{T}_k}\sigma_K^{k}.
\end{equation}
We note that, owing to the shape regularity of the family of meshes $\{\mathcal{T}_k\}_{k\in\mathbb{N}}$, there exist constants $\underline{\sigma},\overline{\sigma}>0$, independent of $k\in\mathbb{N}$, such that
$$\underline{\sigma}\leq\sigma^{k}\leq \overline{\sigma}\quad\forall \,k\in\mathbb{N}.$$

We assume in addition, for Subsection~\ref{sec:scheme}, Subsection~\ref{sec:stab} and Section~\ref{Sec6}, that the family of meshes $\{\mathcal{T}_k\}_{k\in\mathbb{N}}$ is \emph{strictly acute} \cite{burman2002nonlinear} in the following sense: there exists $\theta\in (0,\pi/2)$, independent of $k\in\mathbb{N}$, such that, for each $k\in\mathbb{N}$, the nodal basis $\{\xi_1,\cdots,\xi_{N_k}\}$ of $V_k$ satisfies
\begin{equation}\label{acute}
	\nabla \xi_{i}\cdot\nabla \xi_j|_K\leq -\sin(\theta)\left|\nabla \xi_i|_K\right| \left|\nabla \xi_j|_K\right|\quad \forall 1\leq i,j\leq N_k, i\neq j,\forall K\in\mathcal{T}.
\end{equation}
The condition \eqref{acute} can be interpreted geometrically. For instance, in two space dimensions the strict acuteness condition \eqref{acute} indicates that the largest angle of a given triangle $K\in\mathcal{T}_k$ is at most $\frac{\pi}{2}-\theta$, while in three space dimensions \eqref{acute} indicates that each angle formed by the six pairs of faces of any tetrahedron $K\in\mathcal{T}_k$ is at most $\frac{\pi}{2}-\theta$ (see \cite{burman2002nonlinear}).

\subsection{A Monotone Finite Element Method}\label{sec:scheme}
The proof of uniqueness of weak solutions of \eqref{weakform} uses the Comparison Principle of elliptic operators (see Section~\ref{Sec 4}). In order to preserve this approach on the discrete level, we consider here approximations by a monotone finite element method that satisfies a discrete maximum principle (see, e.g., \cite{ciarlet1973maximum}). As such, we will consider a finite element discretization of \eqref{weakform} that employs the method of artificial diffusion on strictly acute meshes \cite{burman2002nonlinear,JensenSmears2013} to ensure nonnegativity of the approximations for the density.

We introduce a family of artificial diffusion coefficients that will be used in the finite element discretization of \eqref{weakform}. Let $\mu>1$ be a fixed constant.
Then, for each $k\in\mathbb{N}$, we define the artificial diffusion coefficient $\gamma_k:\Omega\to\mathbb{R}$ element-wise over $\mathcal{T}_k$ by 
\begin{equation}\label{eta-iota}
	\gamma_k|_{K}\coloneqq 
	\max\left(\mu\frac{\|b\|_{C(\overline{\Omega}\times\mathcal{A};\mathbb{R}^n)}\text{diam}(K)+\kappa\text{diam}(K)^2}{\sigma^{k}\sin(\theta)}-\nu,0\right)\quad \forall K\in\mathcal{T}_k.
\end{equation}

With the artificial diffusion coefficients $\{\gamma_k\}_{k\in\mathbb{N}}$ given by \eqref{eta-iota}, the P1-continuous Galerkin finite element discretization of \eqref{weakform} that we consider is the following:
\emph{Given $k\in\mathbb{N}$, find $(u_k,m_k)\in V_{k}\times V_{k}$ such that  }\emph{there exists} $\tilde{b}_k\in D_p H[u_k]$ \emph{satisfying} 
\begin{subequations}\label{weakdisc}
	\begin{align}
		\int_{\Omega}(\nu+\gamma_k)\nabla u_k\cdot\nabla \psi+H(x,\nabla u_k)\psi+\kappa u_k\psi\,\mathrm{d}x &=\langle F[m_k],\psi\rangle_{H^{-1}\times H_0^1}&&\forall \psi\in V_k, \label{weakformdisc1}
		\\
		\int_{\Omega}(\nu+\gamma_k)\nabla m_k\cdot\nabla \phi+m_k\tilde{b}_k\cdot\nabla \phi+\kappa m_k\phi\,\mathrm{d}x &=\langle G,\phi\rangle_{H^{-1}\times H_0^1}&&\forall\phi\in V_k. \label{weakformdisc2}
	\end{align}
\end{subequations}

\begin{remark}[Basic Properties of Artificial Diffusion Coefficients]\label{art-diff-coeff-properties}
	\emph{We observe some key properties of the artificial diffusion coefficients given by \eqref{eta-iota}. Firstly, for each $k\in\mathbb{N}$, $\gamma_k$ is in $L^{\infty}(\Omega)$, constant element-wise, and nonnegative a.e.\ in $\Omega$. Moreover, there holds $\sup_{k\in\mathbb{N}}\|\gamma_k\|_{L^{\infty}(\Omega)}<\infty$  due to shape regularity of the family $\{\mathcal{T}_k\}_{k\in\mathbb{N}}$. Secondly, since the sequence of mesh sizes $\{h_k\}_{k\in\mathbb{N}}$ satisfies $h_k\to 0$ as $k\to\infty$ by assumption, there exists $k_{*}\in\mathbb{N}$ that depends on $\nu$, $\mu$, $\theta$, $\kappa$ and $\|b\|_{C(\overline{\Omega}\times\mathcal{A};\mathbb{R}^n)}$ such that, for all $k\geq k_{*}$, we have $\gamma_k=0$ a.e.\ in $\Omega$. Hence,
	\begin{equation}\label{vanishing-art-diff}
		\|\gamma_k\|_{L^{\infty}(\Omega)}=0\quad \forall \,k\geq k_*,
	\end{equation}
	and thus we recover consistency of the discrete problems \eqref{weakdisc} with the continuous problem \eqref{weakform} in the limit as the mesh-size vanishes.
	More generally, it is well-known that the inclusion of artificial diffusion provides sufficient, but not always necessary, conditions for obtaining a discrete maximum principle.}
\end{remark}

\subsection{Main Results}\label{sec:stab}
We can now state the main results concerning the finite element scheme \eqref{weakdisc}. The first result concerns the existence of solutions to \eqref{weakdisc}.
\begin{theorem}[Existence]\label{Full'}
	For each $k\in\mathbb{N}$ there exists a discrete solution pair $(u_k,m_k)\in V_k\times V_k$ that solves \eqref{weakdisc}. Moreover, there exist constants $C_1^*,C_2^*\geq 0$, independent of $k\in\mathbb{N}$, such that
	\begin{subequations}\label{uniform_boundss}
		\begin{align}
			\sup_{k\in\mathbb{N}}\|m_k\|_{H^1(\Omega)}&\leq C_1^*\|G\|_{H^{-1}(\Omega)},\label{mk-unif-bound}
			\\ \sup_{k\in\mathbb{N}}\|u_k\|_{H^1(\Omega)}&\leq C_2^{*}\left(\|G\|_{H^{-1}(\Omega)}+\|f\|_{C(\overline{\Omega}\times\mathcal{A})}+1\right).\label{uk-unif-bound}
		\end{align}
	\end{subequations}
\end{theorem}
Next, uniqueness of solutions to \eqref{weakdisc} holds under the same monotonicity and nonnegativity assumptions of Theorem \ref{thm2-uniqueness}. 
\begin{theorem}[Uniqueness]\label{thm2''} 
	Suppose that the coupling term $F$ is strictly monotone and $G$ is nonnegative in the sense of distributions in $H^{-1}(\Omega)$. Then, for each $k\in\mathbb{N}$, there exists a unique discrete solution pair $(u_k,m_k)\in V_k\times V_k$ to \eqref{weakdisc}.
\end{theorem}
The proofs of Theorem \ref{Full'} and Theorem \ref{thm2''} are given in Section~\ref{Sec6} below.

We now state the first main result on the convergence of the scheme \eqref{weakdisc}.

\begin{theorem}[Convergence]\label{convergence}
	Assume that the coupling term $F$ is strictly monotone and $G$ is nonnegative in the sense of distributions in $H^{-1}(\Omega)$. Let $(u,m)$ denote the unique pair that solves \eqref{sys} in the sense of Definition \ref{weakdef} and let $\{(u_k,m_k)\}_{k\in\mathbb{N}}$ denote the sequence of solutions generated by \eqref{weakdisc}. Then, as $k\to\infty$,
	\begin{align}\label{mainconv}
		u_k\to u\quad\text{\emph{in}}\quad H_0^1(\Omega),\quad m_k\to m\quad\text{\emph{in}}\quad L^q(\Omega),\quad 
		m_k\rightharpoonup m\quad\text{\emph{in}}\quad H_0^1(\Omega),
	\end{align}
	for any $q\in[1,2^*)$ where $2^*=\infty$ if $n=2$ and $2^*=\frac{2n}{n-2}$ if $n\geq 3$.
\end{theorem}

In general, the strong convergence of $\nabla m_k$ to $\nabla m$ in $L^{2}(\Omega;\mathbb{R}^n)$ is not known. The difficulty lies in the fact that it appears possible that the sequence $\{\tilde{b}_k\}_{k\in\mathbb{N}}$ might be such that there is no subsequence that converges in a sufficiently strong sense.
However, under additional conditions, the weak convergence of the density approximations in $H_0^1(\Omega)$  can be improved to strong convergence. This is the case, for instance, if the Hamiltonian $H$ given by \eqref{1''} is such that partial derivative $\frac{\partial H}{\partial p}$ exists and is continuous in $\Omega\times\mathbb{R}^n$. In fact, a weaker hypothesis than this can be formulated that ensures strong convergence of the density approximations in $H_0^1(\Omega)$. 
\begin{corollary}[Strong $H^1$-Convergence for Density Approximations]\label{convergence-cor}
	In addition to the hypotheses of Theorem \ref{convergence}, suppose that the sequence of transport vector fields $\{\tilde{b}_k\}_{k\in\mathbb{N}}$ from \eqref{weakdisc} is pre-compact in $L^1(\Omega;\mathbb{R}^n)$. Then, $m_k$ converges to $m$ strongly in $H_0^1(\Omega)$ as $k\to \infty$.
\end{corollary}
\begin{remark} 
	\emph{The compactness hypothesis on the sequence of transport vector fields introduced above is satisfied when the partial derivative $\frac{\partial H}{\partial p}$ exists and is continuous in $\Omega\times\mathbb{R}^n$. Indeed, in this case 
	$\tilde{b}_k= \frac{\partial H}{\partial p}(x,\nabla u_k)$ in $L^{\infty}(\Omega;\mathbb{R}^n)$ for all $ k\in\mathbb{N}.$ With the strong convergence of the value function approximations in $H_0^1(\Omega)$ guaranteed by Theorem \ref{convergence}, it is easy to see that the entire sequence $\{\tilde{b}_k\}_{k\in\mathbb{N}}$ converges strongly to $\frac{\partial H}{\partial p}(x,\nabla u)$ in $L^s(\Omega;\mathbb{R}^n)$ for any $s\in [1,\infty)$. Hence, the sequence is pre-compact in $L^1(\Omega;\mathbb{R}^n)$.}
\end{remark}

\section{Analysis of Monotone Finite Element Scheme}\label{Sec6}
\subsection{Stabilization of Linear Differential Operators}

We will say that a linear operator $L:V_k\to V_k^*$ satisfies the \emph{discrete maximum principle} provided that the following condition holds: if $w\in V_{k}$ and $\langle Lw,\xi_i\rangle_{V_k^*\times V_k}\geq 0$ for all $ i\in\{1,\cdots,N_k\}$, then $w\geq 0$ in $\Omega$.
For the analysis of \eqref{weakdisc}, 
we introduce a collection of linear operators that are perturbations of discrete relatives to the operators considered in the continuous setting of Lemma \ref{unif}. Recall the definition of the sequence $\{\gamma_k\}_{k\in\mathbb{N}}$ given in \eqref{eta-iota}. For each $k\in\mathbb{N}$, let $W_k$ denote the collection of linear operators $L:V_{k}\to V_{k}^*$ of the form $$\langle L w, v\rangle_{V_k^*\times V_k}\coloneqq \int_{\Omega}(\nu+\gamma_k)\nabla w\cdot\nabla v+\tilde{b}\cdot\nabla w v+\kappa w v\,\mathrm{d}x\quad \forall w,v\in V_k,$$  with $\tilde{b}:\Omega\to\mathbb{R}^n$ denoting a Lebesgue measurable vector field satisfying the uniform bound  $\|\tilde{b}\|_{L^{\infty}(\Omega;\mathbb{R}^n)}\leq \|b\|_{C(\overline{\Omega}\times\mathcal{A};\mathbb{R}^n)}$. 

By adapting the proof of \cite[Theorem 4.2]{burman2002nonlinear} and \cite[Section 8]{JensenSmears2013}, we obtain the following result that will assist with ensuring nonnegativity of finite element approximations $\{m_k\}_{k\in\mathbb{N}}$ of the density function and proving convergence of the scheme~\eqref{weakdisc}. 
\begin{lemma}[Stabilization via Artificial Diffusion]\label{artdiff}
	Given $k\in\mathbb{N}$, each operator $L\in W_k$, and its adjoint $L^*$, satisfy the discrete maximum principle.
\end{lemma}

\subsection{Well-posedness}
First, we establish that the numerical scheme \eqref{weakdisc} is well-posed, i.e.\ it admits a unique numerical solution for each $k\in\mathbb{N}$. Suppose $C_0\geq 0$ is given. Fundamental to the conclusion of Lemma \ref{unif} is the fact that operators from $\mathcal{G}(C_0)$, and their adjoints, are invertible as maps from $H_0^1(\Omega)$ into $H^{-1}(\Omega)$.
We will employ a discrete version of Lemma \ref{unif} in the analysis of \eqref{weakdisc}.
\begin{lemma}\label{unifdiscrete}
	There exists a constant $C_1^*> 0$, independent of $k\in\mathbb{N}$, such that
	\begin{equation}\label{discrete_estimates} 
		\sup_{k\in\mathbb{N}}\sup_{L\in W_k}\max\left\{\left\|L^{-1}\right\|_{\mathcal{L}\left(V_{k}^*,V_{k}\right)},\left\|{L^*}^{-1}\right\|_{\mathcal{L}\left(V_{k}^*,V_{k}\right)}\right\}\leq C_1^*.
	\end{equation}
\end{lemma}

\begin{proof}
	It suffices to show that, for any $k\in\mathbb{N}$ and any operator $L\in W_k$ we have
	\begin{equation}\label{DISC}
		\left\|{L^*}^{-1}\right\|_{\mathcal{L}\left(V_{k}^*,V_{k}\right)}\leq C_1^*,
	\end{equation}
	for some constant $C_1^*>0$ independent of $k\in\mathbb{N}$. Once proved, an application of the Hahn--Banach Theorem allows us to deduce that for any $k\in\mathbb{N}$ and any operator $L\in W_k$ we also have $
	\left\|{L}^{-1}\right\|_{\mathcal{L}\left(V_{k}^*,V_{k}\right)}\leq C_1^*.$

	Let $k\in\mathbb{N}$ and $L\in W_k$ be given. We know by Lemma \ref{artdiff} that the operator $L$ and their adjoint $L^*$ satisfy the discrete maximum principle. Since $V_k$ is a finite dimensional vector space and $L:V_k\to V_k^*$ is a linear map, the discrete maximum principle ensures the invertibility of both $L$ and $L^*$ as maps from $V_k$ to $V_k^*$. Moreover, for each $f\in V_k^*$ the unique solution $\overline{u}_k\in V_{k}$ to the equation $L^*\overline{u}_k=f$ in $V_{k}^*$ satisfies the following Garding inequality for some constant $C^*>0$ that is independent of $k$:
	\begin{equation}\label{Garding'}
		\|\overline{u}_k\|_{H^1(\Omega)}\leq C^*\left(\|f\|_{V_{k}^*}+\|\overline{u}_k\|_{\Omega}\right)\quad\forall\,f\in V_{k}^*.
	\end{equation}
	
	Suppose for contradiction that \eqref{DISC} does not hold. Then for every integer $j\in\mathbb{N}$ there exists an integer $k_j\in\mathbb{N}$ and an operator $L_{j}\in W_{k_j}$ such that $$\left\|{L_{j}^*}^{-1}\right\|_{\mathcal{L}\left(V_{k_j}^*,V_{k_j}\right)}>j,$$ with the sequence $\{k_j\}_{j\in\mathbb{N}}$ strictly increasing. 
	This implies, together with \eqref{Garding'}, that there exist sequences $\{\overline{u}_{j}\}_{j\in\mathbb{N}}\subset H_0^1(\Omega)$,  and $\{f_{j}\}_{j\in\mathbb{N}}\subset V_1^*$ such that 
	\begin{subequations}
		\begin{gather}
			\overline{u}_{j}\coloneqq {L_{j}^*}^{-1}f_{j}\in V_{k_j},\quad  f_{j}\in V_{k_j}^*, \quad \|\overline{u}_{j}\|_{\Omega}=1\quad \forall\, j\in\mathbb{N},\label{ufdef} \\ 
			\|f_{j}\|_{V_{k_j}^*}<\frac{C^*}{j-C^*}\quad\forall \,j>C^*.\label{ineq'}
		\end{gather}
	\end{subequations} 
	In particular, for all $j\in\mathbb{N}$, there holds
	\begin{align}
		\langle L^*_{j}\overline{u}_{j},v\rangle_{V_{k_j}^*\times V_{k_j}}= \int_{\Omega}(\nu+\gamma_{k_j})\nabla \overline{u}_{j}\cdot\nabla v+\overline{u}_{j}\tilde{b}_{j}\cdot\nabla v+\kappa \overline{u}_jv\,\mathrm{d}x=\langle f_{j}, v\rangle_{V_{k_j}^*\times V_{k_j}}\label{eqnn}
	\end{align}
	for all $v\in V_{k_j}$, with some $\tilde{b}_{j}$ satisfying the uniform bound $\|\tilde{b}_j\|_{L^{\infty}(\Omega;\mathbb{R}^n)}\leq \|b\|_{C(\overline{\Omega}\times\mathcal{A};\mathbb{R}^n)}$ by definition of the inclusion $L_j\in W_{k_j}$.
	Garding's inequality \eqref{Garding'} and \eqref{ineq'} imply there exists a constant $C>0$, independent of $j$, such that $\|\overline{u}_{j}\|_{H^1(\Omega)}\leq C$ for all $j\in\mathbb{N}$. Since the sequence $\{\tilde{b}_{j}\}_{j\in\mathbb{N}}\subset L^{\infty}(\Omega;\mathbb{R}^n)$ is uniformly bounded, we may pass to a subsequence without change of notation to get, as $j\to\infty$,
	\begin{subequations}
		\begin{gather}
			\overline{u}_{j} \rightharpoonup \overline{u} \text{ in } H^1_0(\Omega), \quad \overline{u}_{j} \to \overline{u} \text{ in } L^2(\Omega),\label{weakconv} \\ 
			\tilde{b}_{j} \rightharpoonup^* \tilde{b}_* \text{ in } L^{\infty}(\Omega;\mathbb{R}^n).\label{c_weak_star}
		\end{gather}
	\end{subequations} 
	
	Let $v_l\in V_l$ be given for some fixed $l\in\mathbb{N}$. By nestedness subspaces $\{V_{k}\}_{k\in\mathbb{N}}$ in $H_0^1(\Omega)$, we get by \eqref{eqnn} that
	\begin{equation}\label{contra-ID}
		\int_{\Omega}\nu\nabla \overline{u}_{j}\cdot\nabla v_{l}+\overline{u}_{j}\tilde{b}_{j}\cdot\nabla v_{l}+\kappa \overline{u}_jv_l\,\mathrm{d}x+\int_{\Omega}\gamma_{k_j}\nabla \overline{u}_{j}\cdot\nabla v_l\,\mathrm{d}x=\langle f_{j}, v_{l}\rangle_{V_{k_j}^*\times V_{k_j}}
	\end{equation}
	for all $j\in\mathbb{N}$ such that $ k_j\geq l$.
	We have for all $l\in\mathbb{N}$ that 
	$\langle f_{j}, v_{l}\rangle_{V_{k_j}^*\times V_{k_j}}\to 0$ as $j\to\infty$ by \eqref{ineq'}. Moreover, uniform boundedness of the sequence $\{\overline{u}_j\}_{j\in\mathbb{N}}$ in $H_0^1(\Omega)$ and \eqref{vanishing-art-diff} imply that
	\begin{equation*}
		\lim_{j\to\infty}\left|\int_{\Omega}\gamma_{k_j}\nabla \overline{u}_{j}\cdot\nabla v_l\,\mathrm{d}x\right|=0\quad\forall\, l\in\mathbb{N}.
	\end{equation*}
	Recalling \eqref{weakconv}, \eqref{c_weak_star}, we send $j\to \infty$ in \eqref{contra-ID} to then obtain 
	$$\int_{\Omega}\nu\nabla \overline{u}\cdot\nabla v_{l}+\overline{u}\tilde{b}_*\cdot\nabla v_{l}+\kappa \overline{u}v_l\,\mathrm{d}x=0\quad \forall \,l\in\mathbb{N}.$$ Since $l$ was arbitrary, density of the union $\bigcup_{l\in\mathbb{N}}V_{l}$ in $H_0^1(\Omega)$ allows us to conclude
	$$\int_{\Omega}\nu\nabla \overline{u}\cdot\nabla \phi+\overline{u}\tilde{b}_*\cdot\nabla\phi+\kappa \overline{u}\phi\,\mathrm{d}x=0\quad \forall \phi\in H_0^1(\Omega).$$ Thus, $\overline{u}$ solves an elliptic equation with an operator that is the adjoint of an operator from the class $\mathcal{G}(C_0)$ for some constant $C_0$. Lemma \ref{unif} then implies that $\overline{u}=0$ in $H_0^1(\Omega)$ necessarily, contradicting the fact that $\|\overline{u}\|_{\Omega}=1$ which follows from \eqref{ufdef} and \eqref{weakconv}. Hence, \eqref{discrete_estimates} is proved for some constant $C_1^*$ independent of $k\in\mathbb{N}$, as required.
\end{proof}

By using Lemma \ref{unifdiscrete} and Schaefer's fixed point theorem, in a fashion similar to the proof of Lemma \ref{thm1}, we obtain a well-posedness result for the discrete HJB equation in  \eqref{weakdisc}.
\begin{lemma}\label{thm1'}
	Let $k\in\mathbb{N}$ be given. Then, for each $\overline{m}\in L^2(\Omega)$ there exists unique ${u}_k\in V_{k}$ such that
	\begin{equation}\label{strong'}
		\int_{\Omega}(\nu+\gamma_k)\nabla {u}_k\cdot\nabla \psi+H(x,\nabla {u}_k)\psi+\kappa {u}_k\psi\,\mathrm{d}x=\langle F[\overline{m}],\psi\rangle_{H^{-1}\times H_0^1}\text{ }\text{ } \forall\psi\in V_k.
	\end{equation} 
	There exists a constant $C_1^*$ independent of $k\in\mathbb{N}$ such that
	\begin{equation}\label{eq:disc_hjb_a_priori_bound}
		\lVert u_k \rVert_{H^1(\Omega)}\leq C_1^*\left(\lVert \overline{m} \rVert_\Omega + \lVert f \rVert_{C(\overline{\Omega}\times \mathcal{A})}+1\right).
	\end{equation}
	Moreover, the solution ${u}_k$ depends continuously on $\overline{m}$, i.e., if $\left\{\overline{m}_j\right\}_{j\in\mathbb{N}}\subset L^2(\Omega)$ is such that $\overline{m}_j\to \overline{m}$ in $L^2(\Omega)$ as $j\to\infty$, then the corresponding sequence of solutions $\left\{\overline{u}_j\right\}_{j\in\mathbb{N}}\subset V_{k}$ to the equation \eqref{strong'} converges in $V_{k}$ to the solution ${u}_k$ of \eqref{strong'}.
\end{lemma}

\begin{proof}[Proof of Theorem \ref{Full'} and Theorem \ref{thm2''}] 
	Let $k\in\mathbb{N}$ be given. We deduce the existence of discrete solutions to \eqref{weakdisc} using
	Lemma \ref{unifdiscrete} and Lemma \ref{thm1'} in an adaptation of the proof of Theorem \ref{thm1-existence}. 
	The uniform estimate \eqref{mk-unif-bound} follows immediately from the discrete KFP equation \eqref{weakformdisc2} and the uniform bound  \eqref{discrete_estimates}. 
	The bound \eqref{uk-unif-bound} then follows directly from \eqref{mk-unif-bound} and \eqref{eq:disc_hjb_a_priori_bound}.

	Uniqueness of solutions to \eqref{weakdisc} follows when $F$ is strictly monotone and $G$ is nonnegative in the sense of distributions in $H^{-1}(\Omega)$, with the details being similar to the proof of Theorem \ref{thm2-uniqueness} since the space $V_k$ is a subspace of $H_0^1(\Omega)$. Indeed, this follows through since the linear differential operator featuring in the discrete KFP equation of \eqref{weakdisc} is the adjoint of an operator in the class $W_k$. Therefore, we have access to the discrete maximum principle via Lemma \ref{artdiff} that ensures nonnegativity everywhere in $\Omega$ for the density approximation $m_k$.
\end{proof}

\subsection{Convergence}
In this section we employ a compactness argument to prove the convergence result Theorem \ref{convergence} when unique solutions of \eqref{weakform} are ensured under the hypotheses of Theorem \ref{thm2-uniqueness}. Note that, in addition, the following proof yields an alternative method of showing the existence of weak solutions to \eqref{sys} in the sense of Definition \ref{weakdef}.
\begin{proof}[Proof of Theorem \ref{convergence}]
	Let $\{(u_k,m_k)\}_{k\in\mathbb{N}}$ denote the sequence of solutions given by Theorem \ref{thm2''} with associated vector fields $\tilde{b}_k\in D_pH[u_k]$ ($k\in\mathbb{N}$). Since Theorem \ref{Full'} indicates that the sequences $\{m_k\}_{k\in\mathbb{N}}$, $\{u_k\}_{k\in\mathbb{N}}$ are uniformly bounded in $H_0^1(\Omega)$, while $\{\tilde{b}_k\}_{k\in\mathbb{N}}$ is uniformly bounded in $L^{\infty}(\Omega;\mathbb{R}^n)$, we may pass to subsequences, without change of notation, that satisfy as $k\to \infty$
	\begin{subequations}
		\begin{align}
			&m_k\rightharpoonup m \qquad\text{in } H_0^1(\Omega),  &&  m_k\to m \qquad\text{in } L^q(\Omega)\label{eq:mk},
			\\
			&u_k\rightharpoonup u \hspace{1.5ex}\qquad\text{in } H_0^1(\Omega),  && u_k\to u \hspace{1.5ex}\qquad\text{in }L^q(\Omega)\label{eq:uk},
			\\
			& \tilde{b}_k\rightharpoonup^*\tilde{b}_* \qquad\text{in }L^{\infty}(\Omega;\mathbb{R}^n),  && \label{eq:bk}
		\end{align} 
	\end{subequations}
	for some $m,u\in H_0^1(\Omega)$, some $\tilde{b}_*\in L^{\infty}(\Omega;\mathbb{R}^n)$, and for any $q\in [1,2^*)$ where the critical exponent $2^*=\infty$ if $n=2$ and $2^*=\frac{2n}{n-2}$ if $n\geq 3$. Notice in particular that $m_k\tilde{b}_k$ converges weakly to $m\tilde{b}_*$ in $L^2(\Omega;\mathbb{R}^n)$ as $k\to\infty$.
	
	Let $v\in V_{j}$ be given, for some fixed $j\in \mathbb{N}$. Since the sequence $\{m_k\}_{k\in\mathbb{N}}$ satisfies
	$$\int_{\Omega}(\nu+\gamma_k)\nabla m_{k}\cdot\nabla v +m_{k}\tilde{b}_{k}\cdot\nabla v+\kappa m_k v\,\mathrm{d}x=\langle G,v\rangle_{H^{-1}\times H_0^1} \quad\forall \, k\geq j,$$ and we have the convergence given by \eqref{eq:mk} and \eqref{eq:bk}, along with the uniform boundedness of the sequence $\{\nabla m_k\}_{k\in\mathbb{N}}$ in $L^2(\Omega;\mathbb{R}^n)$, and the vanishing of the artificial diffusion coefficients given by \eqref{vanishing-art-diff}, we obtain in the limit as $k\to\infty$
	$$\int_{\Omega}\nu\nabla m\cdot\nabla v +m\tilde{b}_*\cdot\nabla v+\kappa mv\,\mathrm{d}x=\langle G,v\rangle_{H^{-1}\times H_0^1}\quad\forall v\in V_{j}.$$ Since $j$ was arbitrary, density of the union $\bigcup_{j\in\mathbb{N}}V_{j}$ in $H_0^1(\Omega)$ allows us to conclude 
	\begin{equation}\label{m:eqn}
		\int_{\Omega}\nu\nabla m\cdot\nabla \phi +m\tilde{b}_*\cdot\nabla \phi+\kappa m\phi\,\mathrm{d}x=\langle G,\phi\rangle_{H^{-1}\times H_0^1} \quad\forall \phi\in H_0^1(\Omega).
	\end{equation}
	
	Next, observe that the  boundedness of $\{u_k\}_{k\in\mathbb{N}}$ in $H_0^1(\Omega)$ and~\eqref{bounds:growth} imply that the sequence $\{H(\cdot,\nabla u_k)\}_{k\in\mathbb{N}}$ is bounded in $L^2(\Omega)$. Therefore, there exists $g\in L^2(\Omega)$ such that, by passing to a subsequence without change of notation, we have as $k\to\infty$ 
	\begin{equation}\label{weakHconv}
		H(\cdot,\nabla u_k)\rightharpoonup g\quad\text{in}\quad L^2(\Omega).
	\end{equation}
	For fixed  $j\in\mathbb{N}$, the definition of the sequence $\{u_k\}_{k\in\mathbb{N}}$ implies that for all $k\geq j$,
	\begin{equation}\label{uk-j-eqn}
		\int_{\Omega}(\nu+\gamma_k)\nabla u_{k}\cdot\nabla v+H(x,\nabla u_{k})v+\kappa u_k v\,\mathrm{d}x=\langle F[m_k],v\rangle_{H^{-1}\times H_0^1}\quad \forall v\in V_j.
	\end{equation} 
	Therefore, the convergence given by \eqref{eq:mk}, \eqref{eq:uk}, and \eqref{weakHconv}, together with the vanishing of the artificial diffusion coefficients given by \eqref{vanishing-art-diff}, and the uniform boundedness of the sequence $\{\nabla u_k\}_{k\in\mathbb{N}}$ in $L^2(\Omega;\mathbb{R}^n)$, imply that we obtain
	$$\int_{\Omega}\nu\nabla u\cdot\nabla v+gv+\kappa uv\,\mathrm{d}x=\langle F[m],v\rangle_{H^{-1}\times H_0^1}\quad \forall v\in V_{j},$$ 
	after sending $k\to\infty$ in \eqref{uk-j-eqn}. As $j$ was arbitrary, we conclude from above that 
	\begin{equation}\label{ueqnn}
		\int_{\Omega}\nu\nabla u\cdot\nabla \psi+g\psi+\kappa u\psi\,\mathrm{d}x=\langle F[m],\psi\rangle_{H^{-1}\times H_0^1}\quad \forall \psi\in H_0^1(\Omega),
	\end{equation}
	and in particular
	$\|\nabla u\|_{\Omega}^2=\nu^{-1}\left(\langle F[m],u\rangle_{H^{-1}\times H_0^1}-\int_{\Omega}gu\,\mathrm{d}x-\kappa \|u\|_{\Omega}^2\right).$
	
	On the other hand, the definition of $\{u_k\}_{k\in\mathbb{N}}$ gives 
	\begin{equation*}
		\begin{split}
			\nu\|\nabla u_k\|_{\Omega}^2+\int_{\Omega}H(x,\nabla u_k)u_k\,\mathrm{d}x+\kappa \|u_k\|_{\Omega}^2+\int_{\Omega} \gamma_k|\nabla u_k|^2\,\mathrm{d}x=\langle F[m_k],u_k\rangle_{H^{-1}\times H_0^1}
		\end{split}
	\end{equation*}
	for each $k\in\mathbb{N}$. In view of the convergence given by \eqref{eq:mk}, \eqref{eq:uk}, and \eqref{weakHconv}, along with the uniform boundedness of the sequence $\{\nabla u_k\}_{k\in\mathbb{N}}$ in $L^2(\Omega;\mathbb{R}^n)$, the vanishing of the artificial diffusion coefficients given by \eqref{vanishing-art-diff}, and the Lipschitz continuity of the coupling term $F$, we find that
	\begin{equation}\label{rel1}
		\lim_{k\to\infty}\|\nabla u_k\|_{\Omega}^2=\nu^{-1}\left(\langle F[m],u\rangle_{H^{-1}\times H_0^1}-\int_{\Omega}gu\,\mathrm{d}x-\kappa \|u\|_{\Omega}^2\right)=\|\nabla u\|_{\Omega}^2.
	\end{equation}
	Because \eqref{eq:uk} holds, we deduce via \eqref{rel1} that 	
	\begin{equation}\label{rel1'}
		\lim_{k\to\infty}\|u_k\|_{H^1(\Omega)}=\| u\|_{H^1(\Omega)}.
	\end{equation}
	Since $u_k$ converges weakly to $u$ in $H_0^1(\Omega)$ by \eqref{eq:uk} and we have convergence of norms by \eqref{rel1'}, we deduce strong convergence in $H_0^1(\Omega)$: $u_k\to u$ as $k\to \infty$. Because the mapping $v\mapsto H(\cdot,\nabla v)$ is Lipschitz continuous from $H^1(\Omega)$ into $L^2(\Omega)$, it then follows that \eqref{weakHconv} is in fact strong convergence in $L^2(\Omega)$ with $g=H(\cdot,\nabla u)$. Hence, \eqref{ueqnn} gives the weak HJB equation:
	\begin{equation}\label{u:eqn}
		\int_{\Omega}\nu\nabla u\cdot\nabla \psi+H(x,\nabla u)\psi+\kappa u\psi\,\mathrm{d}x=\langle F[m],\psi\rangle_{H^{-1}\times H_0^1} \quad \forall \psi\in H_0^1(\Omega).
	\end{equation}
	
	To deduce convergence to a weak solution of \eqref{sys}, we need to show that $\tilde{b}_*\in D_pH[u]$ in addition to the established equations \eqref{m:eqn}, \eqref{u:eqn}. But this follows by applying Lemma \ref{inclusion} to $\{(\tilde{b}_k,u_k)\}_{k\in\mathbb{N}}$ after passing to an appropriate subsequence.
	
	In summary, we have shown that a subsequence of the finite element approximations $\{(u_k,m_k)\}_{k\in\mathbb{N}}$ converges to a weak solution $(u,m)$ of \eqref{sys} in the sense that
	\begin{align}\label{convresult}
		\quad 
		u_k\to u\quad\text{in}\quad H_0^1(\Omega),\quad m_k\to m\quad\text{in}\quad L^q(\Omega),\quad 
		m_k\rightharpoonup m\quad\text{in}\quad H_0^1(\Omega),
	\end{align}
	as $k\to\infty$, for any $q\in [1,2^*)$. But uniqueness of the solution of~\eqref{weakform} then implies that the whole sequence $\{(u_k,m_k)\}_{k\in\mathbb{N}}$ converges to the unique solution of~\eqref{weakform}, as required.
\end{proof}

Under the additional hypothesis on the transport vector fields given in Corollary \ref{convergence-cor} we obtain strong convergence of the density approximations in $H_0^1(\Omega)$.
\begin{proof}[Proof of Corollary \ref{convergence-cor}]
	To prove strong convergence of $\{m_k\}_{k\in\mathbb{N}}$ in $H_0^1(\Omega)$, let us consider an arbitrary subsequence $\{m_{k_j}\}_{j\in\mathbb{N}}$ with corresponding subsequence of transport vector fields
	$\{\tilde{b}_{k_j}\}_{j\in\mathbb{N}}$. By the stated hypotheses in the Corollary \ref{convergence-cor}, together with H\"older's inequality and the fact that the sequence $\{\tilde{b}_{k_j}\}_{j\in\mathbb{N}}$ is uniformly bounded in $L^{\infty}(\Omega;\mathbb{R}^n)$, we deduce that there exists $s>n$ such that the sequence $\{\tilde{b}_{k_j}\}_{j\in\mathbb{N}}$ is pre-compact in $L^s(\Omega;\mathbb{R}^n)$. Therefore, there exists a subsequence of $\{\tilde{b}_{k_j}\}_{j\in\mathbb{N}}$, to which we pass without change of notation, that converges in $L^s(\Omega;\mathbb{R}^n)$, $s>n$, to a Lebesgue measurable vector field $\overline{b}:\Omega\to\mathbb{R}^n$ that is in $L^{\infty}(\Omega;\mathbb{R}^n)$.
	From Theorem \ref{convergence}, we know that $m_{k_{j}}\to m$ in $L^r(\Omega)$ as $j\to\infty$ for any $r\in [1,2^*)$ (where we recall $2^*=\frac{2n}{n-2}$ when $n \geq 3$ and $2^*=\infty$ when $n=2$). Hence, by H\"older's inequality and a suitable choice of $r\in [1,2^*)$, we deduce that $m_{k_{j}}\tilde{b}_{k_{j}}\rightarrow m\overline{b}$ strongly in $L^2(\Omega;\mathbb{R}^n)$ as $j\to\infty$.
	Consequently, the weak convergence of $\{m_{k_{j}}\}_{j\in\mathbb{N}}$ to $m$ in $H_0^1(\Omega)$ and hence weak convergence of the gradients $\{\nabla m_{k_{j}}\}_{j\in\mathbb{N}}$ to $\nabla m$ in $L^2(\Omega)$ implies that
	\begin{equation}\label{key-m-conv-term}
		\lim_{j\to\infty}\int_{\Omega}m_{k_{j}}\tilde{b}_{k_{j}}\cdot\nabla m_{k_{j}}\,\mathrm{d}x = \int_{\Omega}m\overline{b}\cdot\nabla m\,\mathrm{d}x.
	\end{equation}
	From the discrete KFP equation \eqref{weakformdisc2}, we obtain
	\begin{equation}\label{disc-grad-m-eqn}
		\begin{split}
			\int_{\Omega}& (\nu+\gamma_{k_{j}})|\nabla m_{k_{j}}|^2\,\mathrm{d}x+\int_{\Omega}m_{k_{j}}\tilde{b}_{k_{j}}\cdot\nabla m_{k_{j}}\,\mathrm{d}x+\kappa \|m_{k_{j}}\|_{\Omega}^2=\langle G,m_{k_{j}}\rangle_{H^{-1}\times H_0^1}.
		\end{split}
	\end{equation}
	We deduce from \eqref{eq:mk},  \eqref{mk-unif-bound}, \eqref{vanishing-art-diff}, and \eqref{disc-grad-m-eqn} that
	\begin{equation}\label{disc-grad-m-lim}
		\lim_{j\to\infty}\|\nabla m_{k_{j}}\|_{\Omega}^2=\nu^{-1}\left(\langle G,m\rangle_{H^{-1}\times H_0^1}-\kappa \|m\|_{\Omega}^2-\int_{\Omega}m\overline{b}\cdot\nabla m\,\mathrm{d}x\right).
	\end{equation} We also deduce from the discrete KFP equation \eqref{weakformdisc2} that $\overline{b}$ and $m$ satisfy
	$$\int_{\Omega}\nu\nabla m\cdot\nabla \phi+m\overline{b}\cdot\nabla \phi+\kappa m\phi\text{ }\mathrm{d}x =\langle G,\phi\rangle_{H^{-1}\times H_0^1}\quad\forall\phi\in H_0^1(\Omega).$$ Hence,
	\begin{equation}\label{cont-grad-m-eqn}
		\|\nabla m\|_{\Omega}^2=\nu^{-1}\left(\langle G,m\rangle_{H^{-1}\times H_0^1}-\kappa \|m\|_{\Omega}^2-\int_{\Omega}m\overline{b}\cdot\nabla m\,\mathrm{d}x\right).
	\end{equation}
	We thus obtain from \eqref{disc-grad-m-lim} and \eqref{cont-grad-m-eqn}, together with the fact that $m_{k_{j}}\to m$ in $L^2(\Omega)$ as $j\to\infty$, that
	\begin{equation}\label{disc-m-H1-norm-lim}
		\lim_{j\to\infty}\|m_{k_{j}}\|_{H^1(\Omega)}=\| m\|_{H^1(\Omega)}.
	\end{equation}
	Since $m_{k_{j}}$ converges weakly to $m$ in $H_0^1(\Omega)$ by Theorem \ref{convergence} and we have convergence of norms by \eqref{disc-m-H1-norm-lim}, we deduce that $m_{k_{j}}\to m$ as $j\to \infty$ strongly in $H_0^1(\Omega)$. The argument above shows that any subsequence of $\{m_k\}_{k\in\mathbb{N}}$ has a further subsequence that converges to $m$ in $H_0^1(\Omega)$, and thus the whole sequence is convergent. This completes the proof.
\end{proof}

\begin{remark}[Convergence in H\"older Norms]
	\emph{When the space dimension $n=2$, the domain $\Omega$ is convex and $G\in L^2(\Omega)$, one can derive a uniform H\"older-norm bound for the approximating sequence  $\{m_k\}_{k\in\mathbb{N}}$ (see \cite[Theorem 3.20]{ko2018finite}). It follows that $\{m_k\}_{k\in\mathbb{N}}$ converges strongly to $m$ in some H\"older space. This likewise holds for the corresponding sequence of value function approximations $\{u_k\}_{k\in\mathbb{N}}$ if $F:H_0^1(\Omega)\to L^2(\Omega)$.} 
\end{remark} 

\section{Numerical Experiments}\label{Sec 7}
As illustrated by the example in Section~\ref{sec:example}, when the Hamiltonian $H$ is nonsmooth, the solution $(u,m)$ of the problem is not necessarily smooth even in the interior of the domain.
The numerical experiments shown below are designed to study the performance of the method on problems with nonsmooth solutions. 
We also consider relaxing the condition on the meshes to being weakly acute rather than strictly acute, and we investigate the singularly perturbed limit $\nu\rightarrow 0$.
Concerning terminology, for a given problem and a given sequence of meshes $\{\mathcal{T}_k\}_{k\in\mathbb{N}}$, we say that the numerical method has \emph{optimal rates of convergence} in some norm if the rate of convergence of the approximations is the same as the rate of convergence of a sequence of best-approximations from the corresponding approximation spaces $\{V_k\}_{k\in\mathbb{N}}$, once the mesh size is sufficiently small. The optimal rate is naturally dependent on the regularity of the solution of the given problem.

\subsection{Set-up of the First Two Experiments}
For the experiments given in sections~\ref{sec:numexp_1} and~\ref{sec:numexp_2}, we take~$\Omega\subset \mathbb{R}^2$ to be the unit square, and we consider the continuous problem \eqref{weakform} where we let the diffusion coefficient $\nu=1$ and the reaction coefficient $\kappa=0$.
The choice of Hamiltonian $H:\overline{\Omega}\times\mathbb{R}^2\to\mathbb{R}$ for both experiments is set to be
\begin{equation}\label{H-test}
	H(x,p)\coloneqq\max_{\alpha\in \overline{B_1(0)}}\left(\alpha\cdot p\right)=|p|\quad\forall (x,p)\in\overline{\Omega}\times\mathbb{R}^2.
\end{equation}
It is clear from \eqref{H-test} that $b(x,{\alpha})\coloneqq \alpha$ for $(x,\alpha) \in \overline{\Omega}\times\overline{B_1(0)}\subset \mathbb{R}^2$ and hence $\|b\|_{C(\overline{\Omega}\times\mathcal{A};\mathbb{R}^2)}=1$. Moreover, the subdifferential  $\partial_pH:\Omega\times\mathbb{R}^2\rightrightarrows \mathbb{R}^2$ is given by 
\begin{equation}\label{subdiffformula}
	\partial_pH(x,p)=
	\begin{cases}
		\left\{\frac{1}{|p|}p\right\}&\text{if } p\neq 0,
		\\
		\overline{B_1(0)} &\text{if } p=0. 
	\end{cases}
	\quad\forall x \in \Omega.
\end{equation}
The choices of the coupling term $F$ and source $G\in H^{-1}(\Omega)$ differ between both experiments.
The computations are performed on a sequence of uniform, conforming, shape-regular, weakly acute meshes on $\Omega$. The formula~\eqref{eta-iota} shows that the inclusion of artificial diffusion is not always necessary when $\nu$ is large enough on strictly acute meshes. In order to test this also for weakly acute meshes, we take the artificial diffusion coefficient to be identically zero.
{To compute the discrete solution on a given mesh, we employ a fixed-point approximation of \eqref{weakdisc} that is based on solving the discrete HJB equation and discrete KFP equation alternately until convergence.}
In each iteration we approximate the discrete HJB equation \eqref{weakformdisc1} via a policy iteration method and we resolve the linear system resulting from the discrete KFP equation \eqref{weakformdisc2} via LU factorization.
For an introduction to policy iteration in general we refer the reader to \cite{bellman1954theory,howard1960dynamic}. We used the open-source finite element software Firedrake \cite{rathgeber2016firedrake} to perform the computations.

\begin{figure}[tb]
	\centering
	\begin{tabular}{c c} 
		\begin{subfigure}[b]{0.45\textwidth}
			\begin{adjustbox}{width=\linewidth}
				\begin{tikzpicture}
					\begin{loglogaxis}[
						title={Value function $H^1$-norm},
						xlabel={mesh size $h$},
						ylabel={relative error},
						xmax=1.6,
						ymax=1.6,
						legend pos=north west,
						ymajorgrids=true,
						grid style=dashed,
						]
						
						\addplot[
						color=blue,
						mark=square,]
						coordinates {(0.7071068,0.8566275)
							(0.3535534,0.5972302)
							(0.1767767,0.3989674)
							(0.08838835,0.2695634)
							(0.04419417,0.1852135)
							(0.02209709,0.128568)
							(0.01104854,0.08950892)
							(0.005524272,0.06202472)
							(0.002762136,0.04228988)
							(0.001381068,0.02766006)
						};
						
						\logLogHalfSlopeTriangle{0.003}{0.02766006}				
					\end{loglogaxis}  
				\end{tikzpicture}
			\end{adjustbox}
		\end{subfigure}
		&
		\begin{subfigure}[b]{0.45\textwidth}
			\begin{adjustbox}{width=\linewidth} 
				\begin{tikzpicture}
					\begin{loglogaxis}[
						title={Transport vector $L^2$-norm},
						xlabel={mesh size $h$},
						ylabel={relative error},
						xmax=1.6,
						ymax=1.6,
						legend pos=north west,
						ymajorgrids=true,
						grid style=dashed,
						]
						
						\addplot[
						color=blue,
						mark=square,]
						coordinates {(0.7071068,0.8308818)
							(0.3535534,0.6119201)
							(0.1767767,0.2581745)
							(0.08838835,0.1448242)
							(0.04419417,0.07989786)
							(0.02209709,0.04542067)
							(0.01104854,0.02205853)
							(0.005524272,0.01171617)
							(0.002762136,0.006294717)
							(0.001381068,0.003145837)
						};
						\IbiglogLogSlopeTriangle{0.003}{0.003145837}
						
					\end{loglogaxis}  
				\end{tikzpicture}
			\end{adjustbox}
		\end{subfigure}
		\\
		\begin{subfigure}[b]{0.45\textwidth}
			\begin{adjustbox}{width=\linewidth} 
				\begin{tikzpicture}
					\begin{loglogaxis}[
						title={Density function $L^2$-norm},
						xlabel={mesh size $h$},
						ylabel={relative error},
						xmax=1.6,
						ymax=1.6,
						legend pos=north west,
						ymajorgrids=true,
						grid style=dashed,
						]
						
						\addplot[
						color=blue,
						mark=square,]
						coordinates {(0.7071068,0.6179921)
							(0.3535534,0.2044256)
							(0.1767767,0.05467305)
							(0.08838835,0.01407674)
							(0.04419417,0.003613242)
							(0.02209709,0.0009656856)
							(0.01104854,0.0002311831)
							(0.005524272,6.104412e-05)
							(0.002762136,1.69423e-05)
							(0.001381068,5.89055e-06)
						};
						\biglogLogTSlopeTriangle{0.003}{5.89055e-06}
					\end{loglogaxis}  
				\end{tikzpicture}
			\end{adjustbox}
		\end{subfigure} 
		&
		\begin{subfigure}[b]{0.45\textwidth}
			\begin{adjustbox}{width=\linewidth} 
				\begin{tikzpicture}
					\begin{loglogaxis}[
						title={Density function $H^1$-norm},
						xlabel={mesh size $h$},
						ylabel={relative error},
						xmax=1.6,
						ymax=1.6,
						legend pos=north west,
						ymajorgrids=true,
						grid style=dashed,
						]
						
						\addplot[
						color=blue,
						mark=square,]
						coordinates {(0.7071068,0.7256666)
							(0.3535534,0.3957164)
							(0.1767767, 0.1983818)
							(0.08838835,0.09951306)
							(0.04419417,0.04977911)
							(0.02209709,0.02502458)
							(0.01104854,0.01237219)
							(0.005524272,0.006134721)
							(0.002762136,0.003015607)
							(0.001381068,0.001414693)
						};
						\IbiglogLogSlopeTriangle{0.003}{0.001414693}
						
					\end{loglogaxis}  
				\end{tikzpicture}
			\end{adjustbox}
		\end{subfigure} 
	\end{tabular} 
	\caption{First experiment -- convergence plots for approximations of the value function, density function, and transport vector. The rate of convergence for $H^1$-norms of the errors of the approximations of the value function is close to the optimal value of $1/2$, and the rate of convergence in the $H^1$-norm for the density function is of order $1$.}
	\label{exp-1:rel-err-plots}
\end{figure}
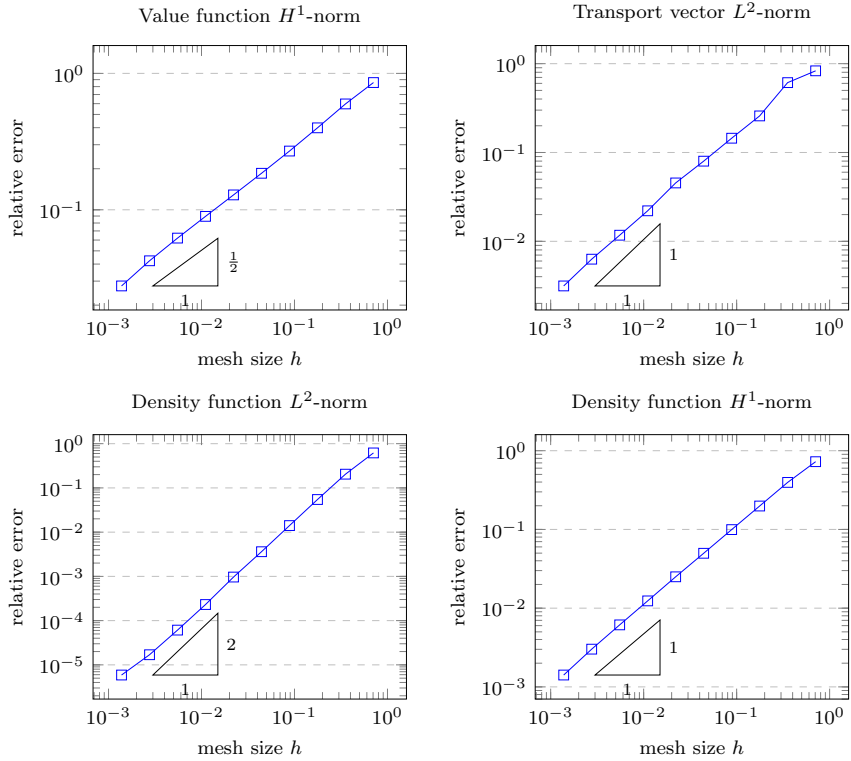

\subsection{First Experiment}\label{sec:numexp_1}
We consider the approximation of a known solution pair that uniquely satisfies \eqref{weakform} with suitable coupling term $F$ and source term $G\in H^{-1}(\Omega)$. 
For this experiment we take the coupling term $F:L^2(\Omega)\to H^{-1}(\Omega)$ defined via
\begin{equation*}\label{F-choice}
	\langle F[v],\psi\rangle_{H^{-1}\times H_0^1}\coloneqq \int_{\Omega}\tanh(v)\psi\,\mathrm{d}x+\langle J,\psi\rangle_{H^{-1}\times H_0^1}\quad\forall \psi\in H_0^1(\Omega),\forall v\in L^2(\Omega),
\end{equation*}
where the functional $J\in H^{-1}(\Omega)$ is given by
\begin{equation*}\label{J-choice}
	\langle J,\psi\rangle_{H^{-1}\times H_0^1}\coloneqq \int_{\Omega}h\psi+\tilde{t}\cdot\nabla \psi\,\mathrm{d}x\quad\forall \psi\in H_0^1(\Omega),
\end{equation*}
with $\tilde{t}:\Omega\to\mathbb{R}^2$ and $h:\Omega\to\mathbb{R}$ defined by
$$\tilde{t}(x,y)\coloneqq\begin{pmatrix}(1+\log(x))y\log(y),\\(1+\log(y))x\log(x)\end{pmatrix},\quad h(x,y)\coloneqq\left|\tilde{t}(x,y)\right|-\tanh\left(xy(1-x)(1-y)\right).$$ Note that $\tilde{t}\in L^2(\Omega;\mathbb{R}^2)$ and $h\in L^2(\Omega)$ so indeed $J\in H^{-1}(\Omega)$.
Next, we take the source term $G\in H^{-1}(\Omega)$ to be 
\begin{equation*}\label{G-choice}
	\langle G,\phi\rangle_{H^{-1}\times H_0^1}\coloneqq \int_{\Omega}r\phi +\tilde{g}\cdot\nabla\phi\,\mathrm{d}x \quad \forall \phi\in H_0^1(\Omega),
\end{equation*}
where $r:\Omega\to\mathbb{R}$ is given by $r(x,y)\coloneqq 2(x(1-x)+y(1-y))$ and $\tilde{g}:\Omega\to\mathbb{R}^2$ is the vector field whose image is $\frac{xy(1-x)(1-y)}{\left|\tilde{t}\right|}\tilde{t}$ whenever $\tilde{t}(x,y)\neq (0,0)^T$ and is  $(0,0)^T$ whenever $\tilde{t}(x,y)=(0,0)^T$.
It can be shown that $F$ is strictly monotone and that $\langle G,\phi\rangle_{H^{-1}\times H_0^1}\geq 0$ if $\phi\in H_0^1(\Omega)$ is nonnegative a.e.\ in $\Omega$. Therefore, the weak formulation~\eqref{weakform} is indeed well-posed according to Theorem \ref{thm2-uniqueness}.
Moreover, the unique solution $(u,m)$ to \eqref{weakform} in this case is given by 
\begin{equation}\label{exact_pair}
	u(x,y)\coloneqq xy\log(x)\log(y),\quad m(x,y)\coloneqq xy(1-x)(1-y),\quad \forall (x,y)\in \Omega.
\end{equation}

In Figure \ref{exp-1:rel-err-plots} we plot the relative errors in various norms versus the mesh-size $h$ for a sequence of finite element approximations obtained from \eqref{weakdisc}. 
Observe that convergence in norm is seen in each plot.
We see that the $H^1$-norm of the error of the approximations of the value function converge at a slower rate than that of the error of the approximations of the density function.
This is due to the fact that $u$ has lower regularity compared to $m$: $m\in C^{\infty}(\Omega)\cap C(\overline{\Omega})$ but $u\notin H^2(\Omega)$.
In fact, $u$ is in the Besov space $H^{\frac{3}{2}-\epsilon}(\Omega)$ for arbitrarily small $0<\epsilon<\frac{1}{2}$.
In Figure \ref{exp-1:rel-err-plots}, we thus see that the convergence rates of the method are optimal.
Moreover, the transport vector field approximations converge strongly in the $L^2$-norm. This is likely due to the transport vector field approximations converging a.e.\ to the vector field $\frac{1}{m}\tilde{g}$, along with the fact that the gradient $\nabla u\neq 0$ a.e.\ in $\Omega$.

\begin{figure}[tb]
	\centering
	\begin{subfigure}[b]{0.49\textwidth}
		\begin{adjustbox}{width=\linewidth} 
			\includegraphics[width=10cm]{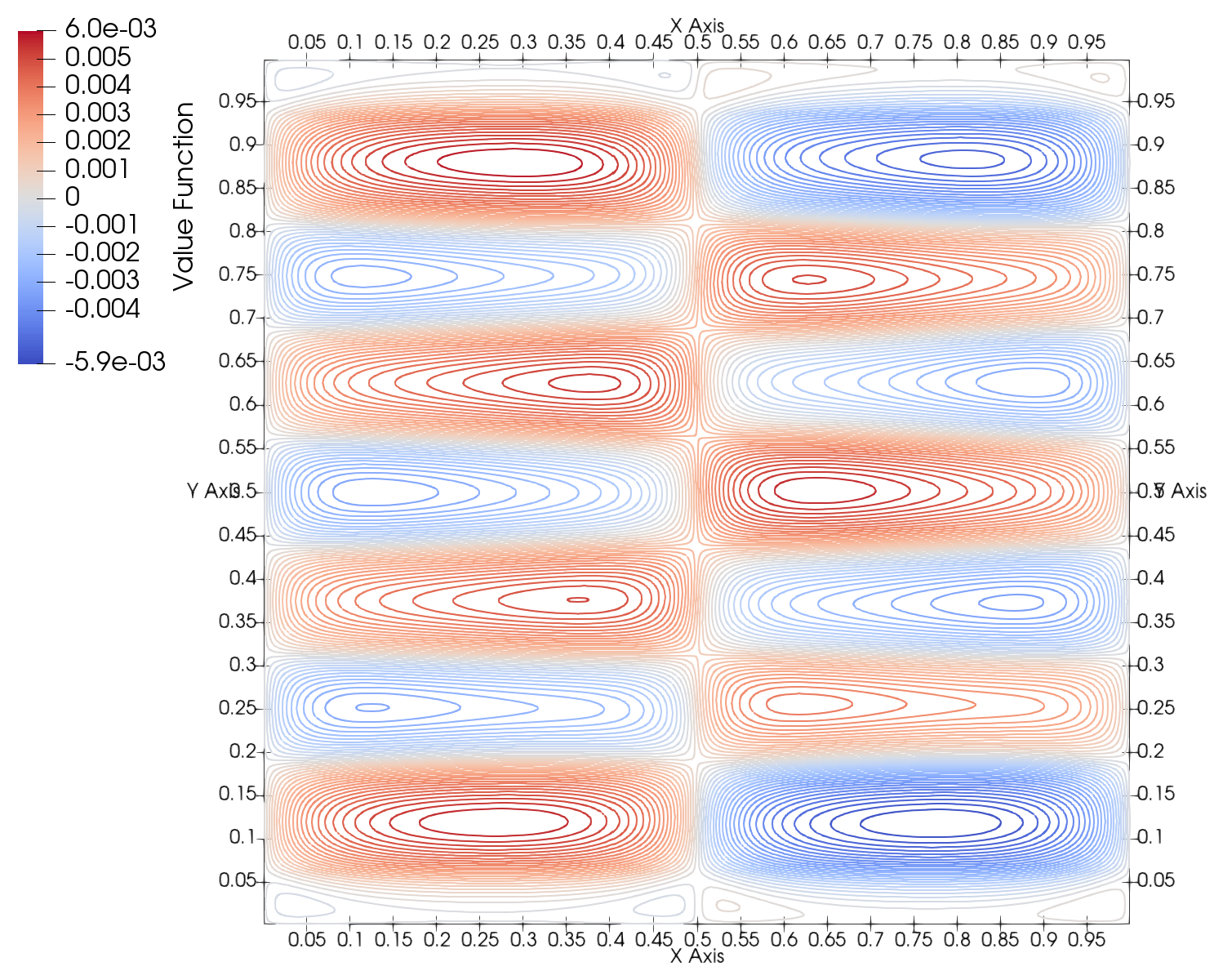}
		\end{adjustbox}
	\end{subfigure} 
	\hfill
	\begin{subfigure}[b]{0.49\textwidth}
		\begin{adjustbox}{width=\linewidth} 
			\includegraphics[width=10cm]{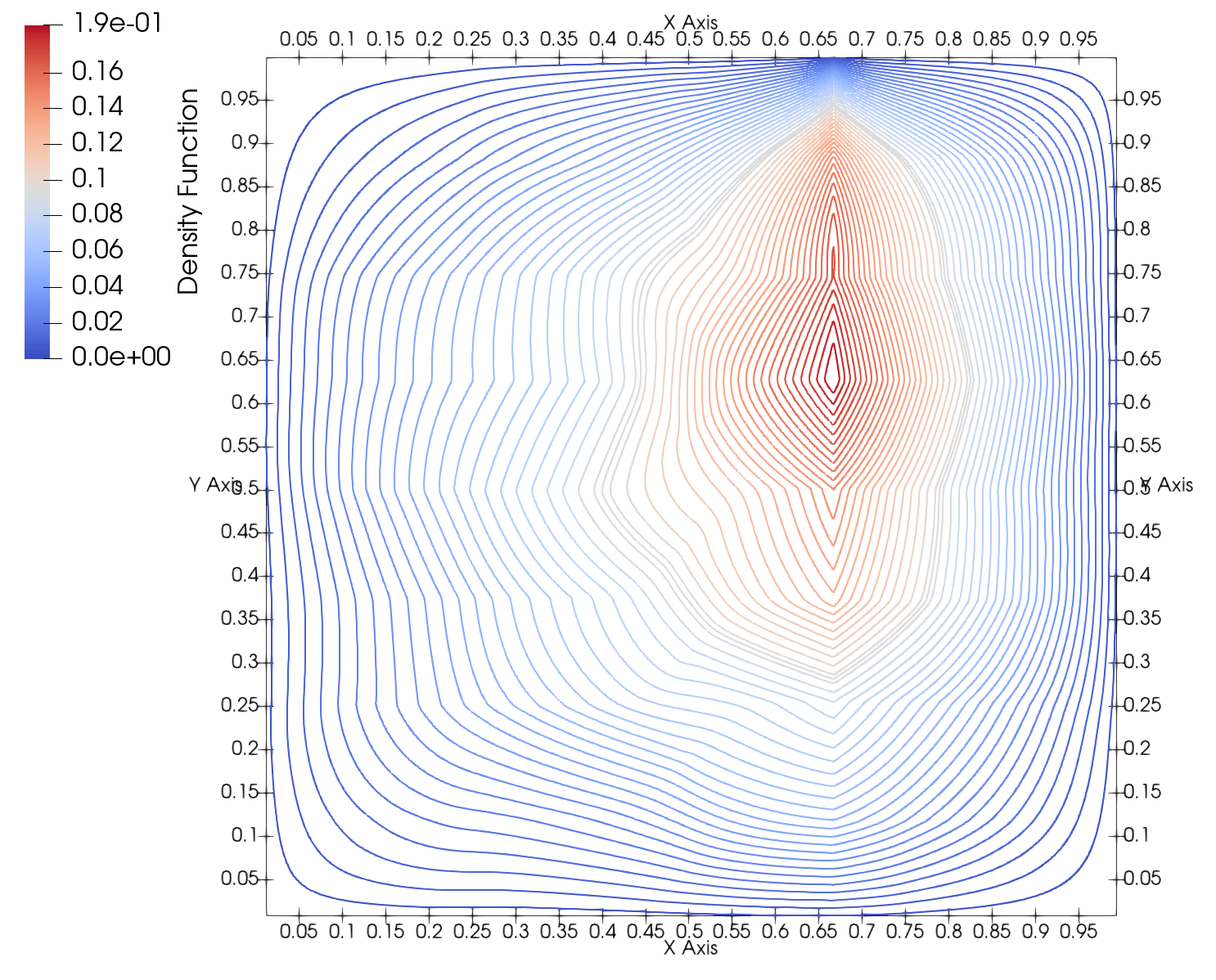}
		\end{adjustbox}
	\end{subfigure} 
	\caption{Second experiment -- approximate contour plot of $u$ (left), $m$ (right) computed on a fine mesh.}
	\label{fig: contour plot reference solution}
\end{figure}	

\begin{figure}[tb]
	\begin{tabular}{c c}
		\begin{subfigure}[b]{0.45\textwidth}
			\begin{adjustbox}{width=\linewidth}
				\begin{tikzpicture}
					\begin{loglogaxis}[
						title={Value function $H^1$-norm},
						xlabel={mesh size $h$},
						ylabel={relative error},
						xmax=1.6,
						ymax=1.6,
						legend pos=north west,
						ymajorgrids=true,
						grid style=dashed,
						]
						
						\addplot[
						color=blue,
						mark=square,]
						coordinates {(0.7071068,0.9994682)
							(0.3535534,1.009994)
							(0.1767767,0.6551388)
							(0.08838835,0.55465)
							(0.04419417,0.2889046)
							(0.02209709,0.1460316)
							(0.01104854,0.0730617)
							(0.005524272,0.03626284)
							(0.002762136,0.01776119)
							(0.001381068,0.008306918)
							
						};
						\IbiglogLogSlopeTriangle{0.003}{0.008306918}
					\end{loglogaxis}  
				\end{tikzpicture}
			\end{adjustbox}
		\end{subfigure}
		&
		\begin{subfigure}[b]{0.45\textwidth}
			\begin{adjustbox}{width=\linewidth} 
				\begin{tikzpicture}
					\begin{loglogaxis}[
						title={Transport vector $L^2$-norm},
						xlabel={mesh size $h$},
						ylabel={relative error},
						xmax=1.6,
						ymax=1.6,
						legend pos=north west,
						ymajorgrids=true,
						grid style=dashed,
						]
						
						\addplot[
						color=blue,
						mark=square,]
						coordinates {(0.7071068,1.313122)
							(0.3535534,1.390952)
							(0.1767767,0.7438815)
							(0.08838835,0.5601214)
							(0.04419417,0.4301933)
							(0.02209709,0.2810091)
							(0.01104854,0.1704523)
							(0.005524272,0.09916042)
							(0.002762136,0.05424726)
							(0.001381068,0.02847871)
							
						};
						\IlogLogSlopeTriangle{0.003}{0.02847871}
					\end{loglogaxis}  
				\end{tikzpicture}
			\end{adjustbox}
		\end{subfigure}
		\\
		\begin{subfigure}[b]{0.45\textwidth}
			\begin{adjustbox}{width=\linewidth} 
				\begin{tikzpicture}
					\begin{loglogaxis}[
						title={Density function $L^2$-norm},
						xlabel={mesh size $h$},
						ylabel={relative error},
						xmax=1.6,
						ymax=1.6,
						legend pos=north west,
						ymajorgrids=true,
						grid style=dashed,
						]
						
						\addplot[
						color=blue,
						mark=square,]
						coordinates {(0.7071068,0.6978075)
							(0.3535534,0.2981201)
							(0.1767767,0.08851999)
							(0.08838835,0.02893447)
							(0.04419417,0.01087881)
							(0.02209709,0.004313171)
							(0.01104854,0.001461434)
							(0.005524272,0.0004724728)
							(0.002762136,0.0001607648)
							(0.001381068,5.233398e-05)
							
						};
						
						\logLogTSlopeTriangle{0.003}{5.233398e-05}
						
					\end{loglogaxis}  
				\end{tikzpicture}
			\end{adjustbox}
		\end{subfigure} 
		&
		\begin{subfigure}[b]{0.45\textwidth}
			\begin{adjustbox}{width=\linewidth} 
				\begin{tikzpicture}
					\begin{loglogaxis}[
						title={Density function $H^1$-norm},
						xlabel={mesh size $h$},
						ylabel={relative error},
						xmax=1.6,
						ymax=1.6,
						legend pos=north west,
						ymajorgrids=true,
						grid style=dashed,
						]
						
						\addplot[
						color=blue,
						mark=square,]
						coordinates {(0.7071068,0.9011915)
							(0.3535534,0.655363)
							(0.1767767,0.4611463)
							(0.08838835,0.3201307)
							(0.04419417,0.2300744)
							(0.02209709,0.160974)
							(0.01104854,0.1124735)
							(0.005524272,0.07765708)
							(0.002762136,0.05296566)
							(0.001381068,0.03471321)
							
						};
						\logLogHalfSlopeTriangle{0.003}{0.03471321}
					\end{loglogaxis}  
				\end{tikzpicture}
			\end{adjustbox}
		\end{subfigure} 
	\end{tabular}
	\caption{Second experiment -- convergence plots for approximations of the value function, density function, and transport vector.
		We observe a first-order rate of convergence for the $H^1$-norms of the error of the approximations of the value function, and a rate of order approximately $1/2$ for the $H^1$-norm errors of the approximations of the density function.}
	\label{exp-2:rel-err-plots}
\end{figure}
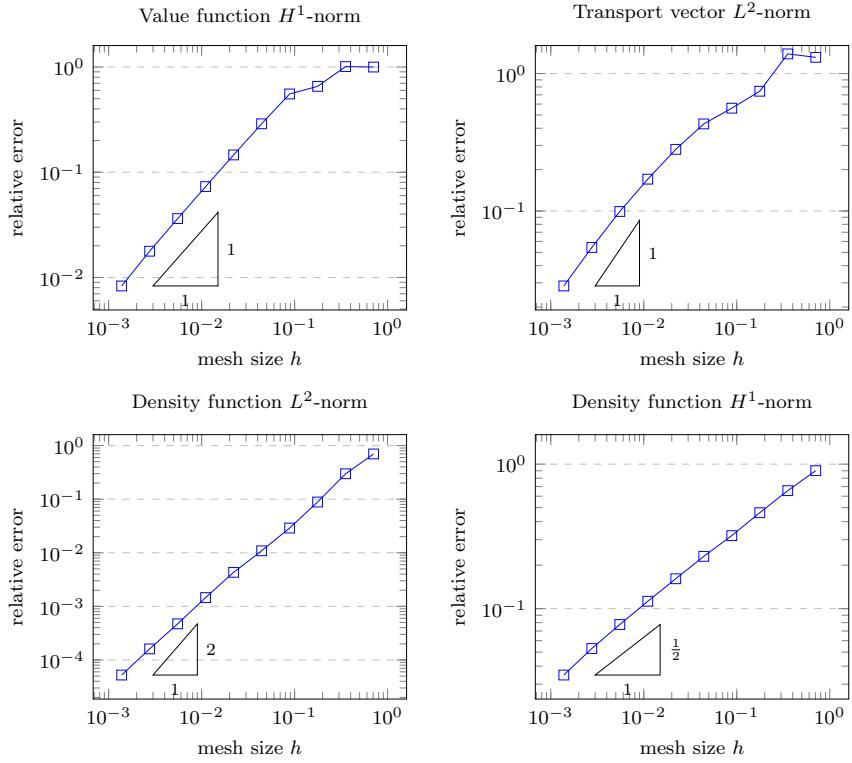

\subsection{Second Experiment}\label{sec:numexp_2}
For this experiment, we choose the coupling term $F$ and source $G\in H^{-1}(\Omega)$ in such a way that the conditions of Theorem \ref{thm2-uniqueness} still hold so that the problem still admits a unique weak solution, but in this example the exact solution is not known explicitly.
We set the coupling term $F:L^2(\Omega)\to H^{-1}(\Omega)$ via $$\langle F[v],\psi\rangle_{H^{-1}\times H_0^1}\coloneqq\int_{\Omega}(\arctan(v)+2\text{sgn}((x-0.5)\cos(8\pi y)))\psi\,\mathrm{d}x$$ for all $\psi\in H_0^1(\Omega)$ and $v\in L^2(\Omega)$, and we let $G\in H^{-1}(\Omega)$ be given by $$\langle G,\phi\rangle_{H^{-1}\times H_0^1}\coloneqq\int_{\Omega}\frac{1}{2}\left(\text{sgn}(\sin(4\pi x)\sin(4\pi y))+1\right)\phi+\tilde{c}\cdot\nabla \phi\,\mathrm{d}x \quad \forall \phi\in H_0^1(\Omega), $$ where $\tilde{c}:\Omega\to\mathbb{R}^2$ is defined by
\begin{equation*}
	\tilde{c}(x,y)\coloneqq 
	\begin{cases}
		y(1,0)^T \quad &\text{if }0<x<2/3,
		\\
		y^2(-1,0)^T \quad &\text{otherwise}.
	\end{cases}
\end{equation*}
It is easy to show that $F$ is strictly monotone and that $\langle G,\phi\rangle_{H^{-1}\times H_0^1}\geq 0$ for all $\phi\in H_0^1(\Omega)$ that is nonnegative a.e.\ in $\Omega$. The exact solution is not known, so we measure the errors using a computed reference solution on a very fine mesh. To illustrate the discrete reference solution that we will consider, in Figure \ref{fig: contour plot reference solution} we display contour plots of the computed reference solution. We note that the displayed approximation for $m$ is nonnegative everywhere in $\Omega$. 
In Figure \ref{exp-2:rel-err-plots} we plot the relative errors versus mesh size $h$. These results were obtained with respect to a discrete reference solution $(u,m)$ that was computed on a much finer mesh than the approximation displayed in Figure \ref{fig: contour plot reference solution}. The convergence rate of the $H^1$-norm of the error in the approximations of the value function approximations is now 1, which is optimal as $u$ is smooth.
The convergence rate of $H^1$-norm errors of the approximation of the density is close to $1/2$, which is the expected optimal rate given the limited regularity of the density function $m$ as a result of the line singularity in the right-hand side source term $G$.

\subsection{Third Experiment} 
For the final experiment, we investigate the robustness of the method in the singularly perturbed limit, where the diffusion parameter $\nu$ becomes small.
Let $\Omega = (-1,1) $, and let $H(x,p)\coloneqq \abs{p}$. Let $F\colon L^2(\Omega)\rightarrow L^2(\Omega)$ be defined by $F[m](x)=m(x)+1$ for all $m\in L^2(\Omega)$, and let $G(x)=1$ for all $x\in \Omega$.
We set $\kappa =0$, and we consider the problem for various values of $\nu>0$. For each $\nu>0$, the exact solution is then given by
\begin{equation}
	\begin{aligned}
		u(x) &= -\abs{x} - \frac{x^2}{2}+ \frac{\nu (\nu+1)}{2} \mathrm{e}^{\frac{\abs{x}-1}{\nu}} + a \mathrm{e}^{-\frac{\abs{x}}{\nu}} + b,\\ 
		m(x) &= \nu+\abs{x}-(\nu+1) \mathrm{e}^{\frac{\abs{x}-1}{\nu}},
	\end{aligned}
\end{equation}  
with constants $a \coloneqq \frac{\nu(\nu+1)}{2}\mathrm{e}^{-\frac{1}{\nu}}-\nu$ and
$b\coloneqq \frac{3-\nu(\nu+1)\big(1+ \mathrm{e}^{-2/\nu}\big)}{2}+ \nu\mathrm{e}^{-\frac{1}{\nu}} $.
Note that the density $m $ is not continuously differentiable, despite all problem data being smooth with the sole exception of the Hamiltonian.
Furthermore, for small $\nu$, the density~$m$ exhibits a sharp boundary layer close to the boundary, and $m(x)\rightarrow m_*(x)\coloneqq \abs{x}$ for all $x\in \Omega$ in the limit $\nu\rightarrow 0$.
Observe that $m_*\notin H^1_0(\Omega)$ since $m_*$ does not satisfy the homogeneous Dirichlet boundary condition, and thus it is clear that we cannot expect the convergence of $m_k$ to $m$ in the $H^1$-norm to be robust with respect to~$\nu$ as $\nu$ becomes small.
Considering instead the $L^2$-norm of the error $\lVert m-m_k\rVert_\Omega$, the expected optimal effective rate of convergence is of order $1/2$ on quasi-uniform meshes in the coarse mesh regime where $\nu \ll h_k $, since it can be proved that $\inf_{v_k\in V_k}\lVert m_* - v_k \rVert_{\Omega}\gtrsim h_k^{1/2}$ for quasi-uniform sequences of meshes.
Given the choice of the coupling term~$F$, we also expect an effective rate of convergence of order $1/2$ for $\lVert u-u_k \rVert_{H^1(\Omega)}$ when $\nu \ll h_k $.

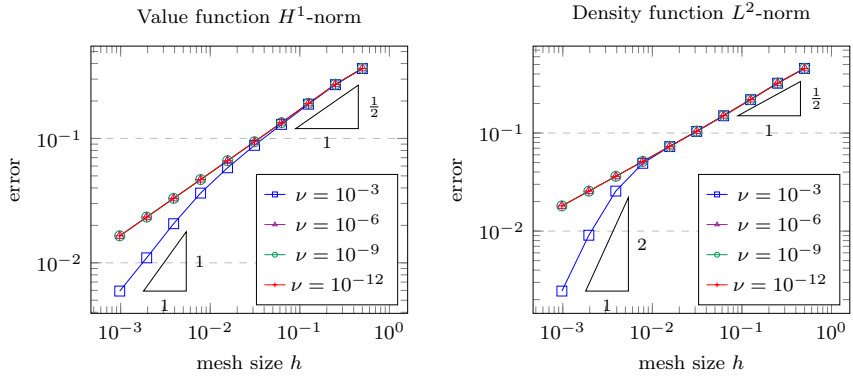
\begin{figure}[tb]
	\begin{tabular}{c c}

		\begin{subfigure}[b]{0.45\textwidth}
			\begin{adjustbox}{width=\linewidth} 
				\begin{tikzpicture}
					\begin{loglogaxis}[title={Value function $H^1$-norm},
						xlabel={mesh size $h$},
						ylabel={error},
						legend pos=south east,
						xmax=1.6,
						ymajorgrids=true,
						grid style=dashed]
						
						\addplot[color=blue,mark=square,]
						coordinates {
							(5.000000000000000e-01,      3.646720395717836e-01)
							(2.500000000000000e-01,      2.705781411675645e-01)
							(1.250000000000000e-01,      1.885252816861381e-01)
							(6.250000000000000e-02,      1.296909544827354e-01)
							(3.125000000000000e-02,      8.804592417210899e-02)
							(1.562500000000000e-02,      5.809020059901655e-02)
							(7.812500000000000e-03,      3.635474626689576e-02)
							(3.906250000000000e-03,      2.062339116500707e-02)
							(1.953125000000000e-03,      1.096481491807301e-02)
							(9.765625000000000e-04,      5.925014692570926e-03)
							
						};
						\IlogLogSlopeTriangle{1.8e-3}{5.925014692570926e-03}
						\logLogHalfSlopeTriangle{9e-2}{1.2e-01}	
						
						\addplot[color=plum,mark=triangle,]
						coordinates {
							(5.000000000000000e-01,      3.659189155845453e-01)
							(2.500000000000000e-01,      2.725909164581162e-01)
							(1.250000000000000e-01,      1.914989947662144e-01)
							(6.250000000000000e-02,      1.339459253190721e-01)
							(3.125000000000000e-02,      9.400400983481011e-02)
							(1.562500000000000e-02,      6.618132546968040e-02)
							(7.812500000000000e-03,      4.668380217136445e-02)
							(3.906250000000000e-03,      3.296253251343390e-02)
							(1.953125000000000e-03,      2.328150798208083e-02)
							(9.765625000000000e-04,      1.643984332916185e-02)
						};	
						\addplot[color=fgreen,mark=o]
						coordinates {
							(5.000000000000000e-01,      3.659201668776643e-01)
							(2.500000000000000e-01,      2.725929397255077e-01)
							(1.250000000000000e-01,      1.915019970321351e-01)
							(6.250000000000000e-02,      1.339502611696746e-01)
							(3.125000000000000e-02,      9.401020322078638e-02)
							(1.562500000000000e-02,      6.619012718826661e-02)
							(7.812500000000000e-03,      4.669627934182971e-02)
							(3.906250000000000e-03,      3.298019700704939e-02)
							(1.953125000000000e-03,      2.330649759780913e-02)
							(9.765625000000000e-04,      1.647517489749288e-02)
						};					
						\addplot[color=red,mark=+,]
						coordinates {
							(5.000000000000000e-01,      3.659201681289618e-01)
							(2.500000000000000e-01,      2.725929417487855e-01)
							(1.250000000000000e-01,      1.915020000344291e-01)
							(6.250000000000000e-02,      1.339502655056034e-01)
							(3.125000000000000e-02,      9.401020941439225e-02)
							(1.562500000000000e-02,      6.619013599060559e-02)
							(7.812500000000000e-03,      4.669629182075258e-02)
							(3.906250000000000e-03,      3.298021467649806e-02)
							(1.953125000000000e-03,      2.330652260144321e-02)
							(9.765625000000000e-04,      1.647521026874800e-02)
						};
						\legend{$\nu=10^{-3}$,$\nu=10^{-6}$,$\nu=10^{-9}$,$\nu=10^{-12}$}
					\end{loglogaxis}
				\end{tikzpicture}
			\end{adjustbox}
		\end{subfigure}
		&
		
		\begin{subfigure}[b]{0.45\textwidth}
			\begin{adjustbox}{width=\linewidth} 
				\begin{tikzpicture}
					\begin{loglogaxis}[
						title={Density function $L^2$-norm},
						xlabel={mesh size $h$},
						xmax=1.6,
						ylabel={error},
						legend pos=south east,
						ymajorgrids=true,
						grid style=dashed]
						
						\addplot[color=blue,mark=square,]
						coordinates {
							(5.000000000000000e-01,      4.557539540878023e-01)
							(2.500000000000000e-01,      3.220596476845037e-01)
							(1.250000000000000e-01,      2.189134627350878e-01)
							(6.250000000000000e-02,      1.499674791135216e-01)
							(3.125000000000000e-02,      1.040084310545418e-01)
							(1.562500000000000e-02,      7.271555681001779e-02)
							(7.812500000000000e-03,      4.913716390664095e-02)
							(3.906250000000000e-03,      2.557886023410398e-02)
							(1.953125000000000e-03,      9.076056737908947e-03)
							(9.765625000000000e-04,      2.446267503933427e-03)	
						};
						\logLogTSlopeTriangle{1.8e-3}{2.446267503933427e-03}	
						\logLogHalfSlopeTriangle{9e-2}{1.499674791135216e-01}				
						
						\addplot[color=plum,mark=triangle,]
						coordinates {
							(5.000000000000000e-01,      4.564347799375795e-01)
							(2.500000000000000e-01,      3.227479182793697e-01)
							(1.250000000000000e-01,      2.194921992342021e-01)
							(6.250000000000000e-02,      1.504113045977318e-01)
							(3.125000000000000e-02,      1.043289608294860e-01)
							(1.562500000000000e-02,      7.299194879533015e-02)
							(7.812500000000000e-03,      5.132590243160483e-02)
							(3.906250000000000e-03,      3.618927899125506e-02)
							(1.953125000000000e-03,      2.555268093990648e-02)
							(9.765625000000000e-04,      1.805532544788332e-02)	
						};
						
						\addplot[color=fgreen,mark=o]
						coordinates {
							(5.000000000000000e-01,      4.564354639029852e-01)
							(2.500000000000000e-01,      3.227486114900419e-01)
							(1.250000000000000e-01,      2.194927859304739e-01)
							(6.250000000000000e-02,      1.504117607745381e-01)
							(3.125000000000000e-02,      1.043292997857388e-01)
							(1.562500000000000e-02,      7.299219464020489e-02)
							(7.812500000000000e-03,      5.132607851040193e-02)
							(3.906250000000000e-03,      3.618940428292818e-02)
							(1.953125000000000e-03,      2.555276978522527e-02)
							(9.765625000000000e-04,      1.805538831853940e-02)
						};	
						
						\addplot[color=red,mark=+]
						coordinates {
							(5.000000000000000e-01, 4.564354645869538e-01)
							(2.500000000000000e-01, 3.227486121832575e-01)
							(1.250000000000000e-01, 2.194927865171781e-01)
							(6.250000000000000e-02, 1.504117612307272e-01)
							(3.125000000000000e-02, 1.043293001247134e-01)
							(1.562500000000000e-02, 7.299219488607660e-02)
							(7.812500000000000e-03, 5.132607868651934e-02)
							(3.906250000000000e-03, 3.618940440827464e-02)
							(1.953125000000000e-03, 2.555276987414875e-02)
							(9.765625000000000e-04, 1.805538838152052e-02)	
						};		
						\legend{$\nu=10^{-3}$,$\nu=10^{-6}$,$\nu=10^{-9}$,$\nu=10^{-12}$}
					\end{loglogaxis}
				\end{tikzpicture}
			\end{adjustbox}
		\end{subfigure}
		
	\end{tabular}
	\caption{Third experiment -- the error  $\lVert m-m_k \rVert_\Omega$ and $\lVert u-u_k \rVert_{H^1(\Omega)}$ for various values of the diffusion coefficient $\nu$. In the coarse-mesh regime $\nu\ll h_k$, the effective rate of convergence is of order $1/2$ for both $\lVert m-m_k \rVert_\Omega$ and $\lVert u-u_k \rVert_{H^1(\Omega)}$ as a result of the strong boundary layer in $m$ that appears in the singularly perturbed limit.}
	\label{fig:third_experiment}
\end{figure}

We apply the method on uniform meshes on $\Omega$ for $\nu\in \{10^{-3},10^{-6},10^{-9},10^{-12}\}$.
Note that in this example, the inclusion of some stabilization, such as artificial diffusion, becomes necessary for stability when $\nu$ is small, so we set the artificial diffusion parameter~$\gamma_k = \max(h_k/2-\nu,0)$.
\cref{fig:third_experiment} shows the errors $\lVert m-m_k \rVert_{\Omega}$ and $\lVert u-u_k\rVert_{H^1(\Omega)}$ for the various values of $\nu$ that are attained on relatively coarse meshes. 
We observe the expected convergence $\lVert m-m_k \rVert_{\Omega}+\lVert u-u_k\rVert_{H^1(\Omega)}\lesssim h_k^{1/2}$ in the regime where $\nu \ll h_k$, which remains robust with respect to $\nu$. 
For the case where $\nu=10^{-3}$, we also observe the final asymptotic convergence rates of order $2$ for $\lVert m-m_k \rVert_{\Omega}$ and order $1$ for $\lVert u-u_k \rVert_{H^1(\Omega)}$ respectively on the finer meshes where $h_k\approx \nu$.

\section{Conclusion}
\label{sec:conclusions}

We have introduced the notion of Mean Field Game Partial Differential Inclusions as a generalization of Mean Field Game systems with possibly nondifferentiable Hamiltonians.
We established the existence and uniqueness of weak solutions of the MFG PDI system \eqref{sys} under appropriate hypotheses.  
We proved the well-posedness and convergence of numerical approximations of the system by a monotone FEM. 
The numerical experiments suggest that the method can achieve optimal rates of convergence for the different solution components $u$ and $m$ even in the case of nonsmooth solution pairs, and also suggest convergence in the small viscosity limit.

\section*{Acknowledgements}
We would like to acknowledge the assistance of Jack Betteridge in enabling the use of Firedrake on the UCL Myriad High Performance Computing Facility (Myriad@UCL). Moreover, the authors acknowledge the use of the UCL Myriad High Performance Computing Facility (Myriad@UCL), and associated support services, in the completion of this work.
\medskip

\appendix
\section{Proofs of Lemmas~\ref{unif} and~\ref{thm1}}\label{appendix-a}

\begin{proof}[Proof of Lemma \ref{unif}]
	We begin by noting that each operator $L\in\mathcal{G}(C_0)$ and its adjoint $L^*$ are invertible owing to the Weak Maximum Principle and the Fredholm Alternative, see \cite[Theorem 8.3]{gilbarg2015elliptic}. Moreover, a standard application of the Hahn-Banach Theorem implies that $L\in\mathcal{G}(C_0)$ satisfies
	\begin{equation}\label{main-ineq}
		\left\|{L^*}^{-1}\right\|_{\mathcal{L}\left(H^{-1}(\Omega),H_0^1(\Omega)\right)}\leq C_1,
	\end{equation}
	for some constant $C_1>0$ if and only if $\left\|{L}^{-1}\right\|_{\mathcal{L}\left(H^{-1}(\Omega),H_0^1(\Omega)\right)}\leq C_1$.
	Therefore, it suffices to prove \eqref{main-ineq}.
	Note that Poincar\'{e}'s inequality implies that there is a constant $C$ depending only on $C_0$ and $\Omega$ such that, if $L\in \mathcal{G}(C_0)$, and $u\in H^1_0(\Omega)$ solves $L^* u = f$ for some $f \in H^{-1}(\Omega)$, then we have the Garding inequality
	\begin{equation}\label{eq:garding}
		\|u\|_{H^1(\Omega)} \leq C \left( \| f \|_{H^{-1}(\Omega)} + \|u\|_{\Omega} \right).
	\end{equation}
	We now prove the result by contradiction, so suppose that~\eqref{main-ineq} is false, i.e.\ there exist sequences $\left\{u_j\right\}_{j\in\mathbb{N}} \subset H^1_0(\Omega)$ and 
	$\left\{f_j\right\}_{j\in\mathbb{N}} \subset H^{-1}(\Omega)$ such that $L_j^* u_j = f_j$ in $H^{-1}(\Omega)$ for some $L_j \in \mathcal{G}(C_0)$, with $\lVert u_j \rVert_{H^1_0(\Omega)}\geq j \lVert f_j \rVert_{H^{-1}(\Omega)}$ for all $j\in\mathbb{N}$. 
	In particular, for each $j\in\mathbb{N}$, we have 
	\begin{equation}\label{eq:L_k}
		\langle L_j^* \,u_j,v\rangle_{H^{-1}\times H_0^1}= \int_{\Omega} \nu \nabla u_j \cdot \nabla v + u_j \tilde{b}_j{\cdot} \nabla v +c_ju_jv\, \mathrm{d}x = \langle f_j,v\rangle_{H^{-1}\times H_0^1}
	\end{equation}
	for all $v\in H_0^1(\Omega)$, for some $\tilde{b}_j \in L^{\infty}(\Omega;\mathbb{R}^n)$ and $c_j \in L^{\infty}(\Omega)$ that satisfy $\|\tilde{b}_j\|_{L^{\infty}(\Omega;\mathbb{R}^n)}+\|c_j\|_{L^{\infty}(\Omega)} \leq C_0$ and $c_j\geq 0$ a.e.\ in $\Omega$.
	Without loss of generality, we can also suppose that $\lVert u_j \rVert_{L^2(\Omega)}=1$ for all $j\in\mathbb{N}$. Then, Garding's inequality~\eqref{eq:garding} implies that, for $j$ sufficiently large, $\|f_j\|_{H^{-1}(\Omega)} \leq\frac{C}{j-C}$ and thus $f_j \rightarrow 0$ in $H^{-1}(\Omega)$ as $j\rightarrow \infty$, as well as $ \sup_{j\in\mathbb{N}}\|u_j\|_{H^1(\Omega)} < \infty$.
	Therefore, there exist subsequences, to which we pass without change of notation, such that $u_j \rightharpoonup u$ in $H^1_0(\Omega)$, $u_j \to u$ in $L^2(\Omega)$, $b_j \rightharpoonup^* \tilde{b}$ in $L^\infty(\Omega;\mathbb{R}^n)$ and $c_j\rightharpoonup^* c$ in $L^{\infty}(\Omega) $ as $j\to \infty$. Note also that the limit $c$ is nonnegative a.e.\ in $\Omega$, which follows from Mazur's Theorem and from the fact that $c_j \rightharpoonup c$ in $L^p(\Omega)$ for all $p<\infty$ as $\Omega$ is bounded.
	By passing to the limit in~\eqref{eq:L_k}, we deduce that 
	\begin{equation}\label{eq:weak_max_contradiction}
		\int\limits_\Omega \nu \nabla u \cdot \nabla v + u\tilde{b}{\cdot} \nabla v +c uv\,\mathrm{d}x = 0 \quad \forall v \in H^1_0(\Omega).
	\end{equation}
	Thus~\eqref{eq:weak_max_contradiction} implies that $u = 0$ a.e.\ in $\Omega$, which contradicts $\|u\|_{\Omega}=\lim_{j\to\infty}\lVert u_j \rVert_{\Omega}=1$, thereby completing the proof.
\end{proof} 

\begin{proof}[Proof of Lemma \ref{thm1}]
	Let $m\in L^2(\Omega)$ be fixed. We start by showing the existence of a solution of~\eqref{weaku}.
	We define the operator $K\colon H^1_0(\Omega)\to H^1_0(\Omega)$ by $v\mapsto K[v]\coloneqq w$ where $w\in H^1_0(\Omega)$ is the unique solution of
	\begin{equation}\label{eq:K_operator}
		\int_{\Omega}\nu\nabla w\cdot\nabla \psi+\kappa w\psi\mathrm{d}x=\langle F[m],\psi\rangle_{H^{-1}\times H_0^1}-\int_{\Omega}H(x,\nabla v)\psi\mathrm{d}x \;\forall \psi\in H_0^1(\Omega).
	\end{equation}
	Note that $u$ solves~\eqref{weaku} if and only if $u=K[u]$ is a fixed point of $K$.
	The Lipschitz continuity of $H$, c.f.\ \eqref{bounds}, and the compactness of the embedding $L^2(\Omega)$ into $H^{-1}(\Omega)$ imply that $K$ is a continuous and compact operator from $H^1_0(\Omega)$ into itself. 
	We now show that the set $\{ w = \lambda K[w], \; 0\leq \lambda \leq 1\}$ is bounded in $H^1_0(\Omega)$, which implies that $K$ satisfies the hypotheses of Schaefer's Fixed Point Theorem~\cite[p.\ 502, Theorem 4]{evans2010partial} and implies the existence of a fixed point $u=K[u]$.
	Note that $w=\lambda K[w]$ for some $0\leq \lambda \leq 1$ if and only if $-\nu\Delta w+\lambda H(\cdot,\nabla w)+\kappa w=\lambda F[m]$ in $H^{-1}(\Omega)$.
	Hence, $w=\lambda K[w]$ implies that there exists $\alpha^* \in \Lambda[w]$ such that $H(x,\nabla w) = b(x,\alpha^*(x)){\cdot}\nabla w(x) - f(x,\alpha^*(x))$ for a.e.\ $x\in \Omega$, and thus $w$ solves the elliptic equation
	\begin{multline*}
		\langle L_{\lambda} w ,\psi\rangle_{H^{-1}\times H_0^1}\coloneqq 
		\int_{\Omega}\nu\nabla w\cdot\nabla \psi+\lambda b(x,{\alpha^*(x)})\cdot\nabla w\psi+\kappa w\psi\mathrm{d}x
		\\ =\int_{\Omega}\lambda f(x,\alpha^*(x))\psi\mathrm{d}x+\langle \lambda F[m],\psi\rangle_{H^{-1}\times H_0^1} \quad \forall \psi \in H^1_0(\Omega).
	\end{multline*}
	Since $\lambda\in [0,1]$, we have $L_{\lambda}\in \mathcal{G}(C_0)$ for a constant $C_0\geq 0$ independent of $\lambda$. Therefore, Lemma~\ref{unif} implies that $\lVert w \rVert_{H^1_0(\Omega)}\leq C_*\left(\|f\|_{C(\mathcal{A}\times\overline{\Omega})}+\|F[m]\|_{H^{-1}(\Omega)}\right) $ for some constant $C_*$ independent of $\lambda$ and $w$.
	This shows that $\{ w = \lambda K[w], \; 0\leq \lambda \leq 1\}$ is bounded and that there exists a solution of \eqref{weaku}.
	A similar argument shows that any solution of \eqref{weaku} satisfies~\eqref{eq:hjb_a_priori_bound} due to the growth condition \eqref{F1} on the coupling term $F$.

	Uniqueness of solutions to \eqref{weaku} is a consequence of the Weak Maximum Principle and the existence of measurable selections of the maximizing set $\Lambda$. Indeed, if $u_1$ and $u_2$ are both solutions to \eqref{weaku}, then there exist measurable selections~$\alpha_i \in \Lambda[u_i]$ for $i\in\{1,2\}$. Then, Lemma~\ref{convv} implies that $b(x,\alpha_1(x)){\cdot}\nabla(u_1 - u_2) \geq H(x,\nabla u_1(x))-H(x,\nabla u_2(x))$ a.e.\ in $\Omega$ for $i,\,j \in\{1,2\}$. Therefore we find that
	\[
	\int_\Omega \nu \nabla(u_1-u_2)\cdot\nabla \psi +b(x,\alpha_1(x)){\cdot}\nabla (u_1-u_2) \psi +\kappa (u_1-u_2)\psi \mathrm{d}x \geq 0
	\]
	for all test functions~$\psi\in H^1_0(\Omega)$ that are nonnegative a.e.\ in $\Omega$.
	Since $u_1-u_2\in H_0^1(\Omega)$, the Weak Maximum Principle of \cite[Theorem 8.1]{gilbarg2015elliptic} then implies that $u_1 \geq u_2$ a.e.\ in $\Omega$.
	By symmetry, we also find that $u_2\geq u_1$ a.e.\ in $\Omega$, thus showing that $u_1=u_2$.

	We now prove the continuous dependence of the solutions on the data. 
	Suppose that we are given a sequence $\left\{m_j\right\}_{j=0}^{\infty}\subset L^2(\Omega)$, with corresponding solutions $\{u_j\}_{j\in\mathbb{N}}$ of~\eqref{weaku}, and suppose that $m_j\to m$ in $L^2(\Omega)$. Let $u\in H^1_0(\Omega)$ denote the corresponding unique solution of~\eqref{weaku} with datum $m$.
	To show convergence of the whole sequence $u_j \to u$, it is enough to show that every subsequence of $\{u_j\}_{j\in\mathbb{N}}$ has a further subsequence that converges strongly to $u$ in $H^1_0(\Omega)$. Let $\{u_{j_k}\}_{k\in\mathbb{N}}$ be an arbitrary subsequence. 
	The a priori bound~\eqref{eq:hjb_a_priori_bound} and the strong convergence of the sequence $\{m_j\}_{j\in\mathbb{N}}$ imply that the sequence $\{u_{j_k}\}_{k\in\mathbb{N}}$ is uniformly bounded in $H_0^1(\Omega)$. 
	Consequently, there exists a subsequence, (to which we pass without change of notation), such that $u_{j_k}\rightharpoonup v$ in $H_0^1(\Omega)$ and $u_{j_k}\to v$ in $L^2(\Omega)$ as $k\to \infty$, for some $v\in H^1_0(\Omega)$. 
	It is easy to see that the sequence $\{h_k\}_{k\in\mathbb{N}}$ defined by $h_{k}:=-H(\nabla u_{j_k})-\kappa u_{j_{k}}$ for $k\geq 0$ is uniformly bounded in $L^2(\Omega)$.
	Passing to a further subsequence without change of notation, we have $h_k \rightharpoonup h$ for some $h\in L^2(\Omega)$.
	Passing to the limit in~\eqref{weaku} with $m_{j_k}\to m$ and $u_{j_k}\rightharpoonup v$, we find that 
	\begin{equation}\label{eq:hjb_v_limit}
		\int_{\Omega}\nu\nabla v\cdot\nabla \psi\mathrm{d}x=\int_{\Omega}h\psi\mathrm{d}x+\langle F\left[m\right],\psi \rangle_{H^{-1}\times H_0^1}\text{ }\text{ }\forall \psi\in H_0^1(\Omega).
	\end{equation}
	Consequently, using weak-times-strong convergence, we find that
	\begin{multline}
		\lim_{k\to \infty}\nu\left\|\nabla \left(v-u_{j_{k}}\right)\right\|_{\Omega}^2
		\\ = \lim_{k\to \infty} \left[\int_{\Omega}\left(h-h_{k}\right)\left(v-u_{j_{k}}\right)\mathrm{d}x+ \langle F[m]-F\left[m_{j_{k}}\right],v-u_{j_{k}}\rangle_{H^{-1}\times H_0^1}\right]
		=0.
	\end{multline}
	Thus $u_{j_{k}}\to v$ in $H_0^1(\Omega)$ as $k\to \infty$. 
	Lipschitz continuity of $H$ and \eqref{eq:hjb_v_limit} then imply that $v$ is a solution of~\eqref{weaku} with datum $m$.
	Thus uniqueness of the solution of~\eqref{weaku} shown above implies that $v=u$ in $H_0^1(\Omega)$.
	As explained above, this shows that the whole sequence $\{u_j\}_{j\in\mathbb{N}}$ converges to $u$ in $H^1_0(\Omega)$.
\end{proof}
\medskip

\printbibliography

\end{document}